\documentclass[10pt]{amsart}
\usepackage{amsfonts}
\usepackage{ifthen}
\usepackage{amsthm}
\usepackage{amsmath,mathrsfs}
\usepackage{graphicx}
\usepackage{amscd,amssymb,amsthm}
\usepackage{color}
\usepackage{hyperref}

\newcounter{minutes}
\divide\time by 60
\newcounter{hours}
\multiply\time by 60 \addtocounter{minutes}{-\time}

\newtheorem{theorem}{Theorem}[section]

\newtheorem{corollary}{Corollary}[section]

\title[Applications of ITCM]{APPLICATIONS OF INTEGRAL TRANSFORMS COMPOSITION METHOD TO
	WAVE--TYPE SINGULAR DIFFERENTIAL EQUATIONS
AND INDEX SHIFT TRANSMUTATIONS }

\author[A. Fitouhi ]
		{Ahmed Fitouhi}
		
		\address{ Ahmed Fitouhi\newline
			University of Tunis El Manar, Tunis, Tunisia}
		\email{Ahmed.Fitouhi@fst.rnu.tn}

\author[I. Jebabli ]
	{In\`{e}ss Jebabli}

\address{ In\`{e}ss Jebabli\newline
	University of Tunis El Manar, Tunis, Tunisia}
\email{jebabli.iness@hotmail.fr}

	\author[E. L. Shishkina ]
	{Elina L. Shishkina}
	
	\address{Elina L. Shishkina \newline
		Voronezh State University,  Voronezh, Russia}
	\email{ilina\_dico@mail.ru}
	
\author[S. M. Sitnik ]
		{Sergei M. Sitnik}
		
		\address{ Sergei M. Sitnik\newline
			Belgorod State National Research University ("BSU"), Belgorod, Russia}
		\email{sitnik@bsu.edu.ru}


\begin{document}

\def\thefootnote{}
\footnotetext{ \texttt{File:~\jobname .tex,
          printed: \number\year-0\number\month-\number\day,
          \thehours.\ifnum\theminutes<10{0}\fi\theminutes}
} \makeatletter\def\thefootnote{\@arabic\c@footnote}\makeatother

\maketitle

\begin{abstract}
In the paper we study applications of integral transforms composition method (ITCM) for obtaining transmutations via integral transforms. It is possible to derive wide range of transmutation operators by this method. Classical integral transforms are involved in the integral transforms composition method (ITCM) as basic blocks, among them are Fourier, sine and cosine--Fourier, Hankel, Mellin, Laplace and some generalized transforms. The ITCM  and transmutations obtaining by it are applied to deriving connection formulas for solutions of singular differential equations and more simple non--singular ones. We consider well--known classes of singular differential equations with Bessel operators, such as classical and generalized Euler--Poisson--Darboux equation and  the generalized radiation problem of A.Weinstein. Methods of this paper are applied to more general linear partial differential equations with Bessel operators, such as multivariate Bessel--type equations, GASPT (Generalized Axially Symmetric Potential Theory)  equations of A.Weinstein,  Bessel--type generalized wave equations with variable coefficients,ultra B--hyperbolic equations and others. So with many results and examples the main conclusion of this paper  is illustrated:  \textbf{\textit{the integral transforms composition method (ITCM) of constructing transmutations is very important and effective tool also for obtaining connection formulas and explicit representations of solutions  to a wide class of singular differential equations, including ones with Bessel operators.}}
\end{abstract}

\maketitle
\numberwithin{equation}{section}
\allowdisplaybreaks

\section{Introduction}\label{sec1}

An important field of applications of differential equations is
the study of linear wave processes.
The classical wave equation on the plane
\begin{equation}\label{Wave1}
\frac{\partial^2 u}{\partial t^2}=a^2\frac{\partial^2 u}{\partial x^2},\qquad u=u(x,t),\qquad t>0,\qquad  a=const
\end{equation}
is  generalized in different directions.
Such generalizations include the telegraph equation and the Helmholtz equation
\begin{equation}\label{Wave2}
\frac{\partial^2 u}{\partial t^2}=a^2\frac{\partial^2 u}{\partial x^2}\pm \lambda^2 u,\qquad a=const, \qquad \lambda=const,
\end{equation}
as well as wave equation with one or two  potential functions $p(t), \ q(x)$
\begin{equation}\label{Wave3}
\frac{\partial^2 u}{\partial t^2}+p(t)u=a^2\frac{\partial^2 u}{\partial x^2}+q(x)u.
\end{equation}
Methods of solution to equations \eqref{Wave1}-\eqref{Wave2} are set out in a large number of classical textbooks and monographs (see, for example, \cite{Curant,Evans}).
For study of equations of the type \eqref{Wave3} the results, techniques and ideas of the methods of transmutation operators were mostly used directly or implicitly (see \cite{Levitan1,Levitan2,Levitan3,Lesha1,Lesha2,Hromov}).

The next step for generalization of wave--type equations (\ref{Wave1}--\ref{Wave3}) is replacing
one or two second derivatives by the Bessel operator (see \cite{kipr})
\begin{equation}\label{Bessel}
B_\nu=\frac{\partial^2}{\partial y^2}+\frac{\nu}{y}\frac{\partial}{\partial y},\qquad y>0,\qquad
\nu=const.
\end{equation}
In this way we obtain singalar PDEs since
 one of the operator coefficients $\frac{\nu}{y}$ tends to infinity in some sense as $x\rightarrow 0$ (see \cite{CSh}).

When Bessel operator acts by space variable $x$ we obtain a {\bf generalisation of wave equation with axial or central symmetry}
\begin{equation}\label{EPDWe}
\frac{\partial^2 u}{\partial t^2}=
\frac{\partial^2 u}{\partial x^2}+\frac{\nu}{x}\frac{\partial u}{\partial x},\qquad u=u(x,t),\qquad x>0,\qquad t\in\mathbb{R},\qquad
\nu=const.
\end{equation}
Representations of the solution of (\ref{EPDWe}) were derived by Poisson in \cite{Poisson1821}.
The initial conditions for (\ref{EPDWe}) have the form
$$
u(x,0)=f(x),\qquad u_t(x,0)=g(t).
$$

When Bessel operator acts by time variable $t$ we get
{\bf Euler--Poisson--Darboux (EPD) equation}
\begin{equation}\label{EPD0}
\frac{\partial^2 u}{\partial t^2}+\frac{\nu}{t}\frac{\partial u}{\partial t}=
a^2\frac{\partial^2 u}{\partial x^2},\qquad u=u(x,t),\qquad t>0,\qquad x\in\mathbb{R},\qquad
a, \ \nu=const.
\end{equation}
EPD equation first appeared  in  Euler's work (see \cite{Euler}, p. 227) and further was studied by Poisson in \cite{Poisson}, Riemann in \cite{Riman} and  Darboux  in \cite{Darboux}.
Initial conditions for (\ref{EPD0}) have the form
$$
u(x,0)=f(x),\qquad t^\nu u_t(x,t)|_{t=0}= g(x).
$$

In the case of many space variables and for $\nu>n-1$ Diaz and Weinberger \cite{Diaz} obtained solutions of the Cauchy problem
\begin{equation}\label{EPD1}
\frac{\partial^2 u}{\partial t^2}+\frac{\nu}{t}\frac{\partial u}{\partial t}=
\sum\limits_{i=1}^n\frac{\partial^2 u}{\partial x_i^2},\qquad u=u(x,t),\qquad t>0,\qquad x\in\mathbb{R}^n,\qquad
\nu=const,
\end{equation}
\begin{equation}\label{UslEPD1}
u(x,0)=f(x),\qquad u_t(x,0)=0.
\end{equation}
For any real $\nu$ problem (\ref{EPD1}--\ref{UslEPD1}) was solved by  Weinstein   (cf. \cite{Weinstein0,Weinstein12,Weinstein13}).

{\bf Generalised Euler--Poisson--Darboux (GEPD) equation}
\begin{equation}\label{EPDWG}
\frac{\partial^2 u}{\partial t^2}+\frac{\nu}{t}\frac{\partial u}{\partial t}=
\frac{\partial^2 u}{\partial x^2}+\frac{k}{x}\frac{\partial u}{\partial x},\qquad u=u(x,t),\qquad t>0,\qquad x>0,\qquad
\nu,k=const
\end{equation}
and its multidimensional generalization
\begin{equation}\label{EPDM}
\frac{\partial^2 u}{\partial t^2}+\frac{\nu}{t}\frac{\partial u}{\partial t}=
\sum\limits_{i=1}^n\left( \frac{\partial^2 u}{\partial x_i^2}+\frac{k_i}{x_i}\frac{\partial u}{\partial x_i}\right),
\end{equation}
$$
 u=u(x_1,..,x_n,t),\qquad t>0,\qquad x_i>0,\qquad
 \nu,k_i=const, \qquad i=1,..,n
$$
were considered in \cite{CSh,Fox,LPSh1,LPSh2,ShSitPul,Smirnov0,Smirnov1,Tersenov}.

EPD equations with spectral parameters
\begin{equation}\label{EPDSP}
\frac{\partial^2 u}{\partial t^2}+\frac{\nu}{t}\frac{\partial u}{\partial t}=
\frac{\partial^2 u}{\partial x^2}\pm\lambda^2 u,\qquad \lambda\in\mathbb{R}
\end{equation}
were studied in \cite{Bresters2,Smirnov0} and GEPD equation with spectral parameter
\begin{equation}\label{EPDSP1}
\frac{\partial^2 u}{\partial t^2}+\frac{\nu}{t}\frac{\partial u}{\partial t}=
\sum\limits_{i=1}^n\left( \frac{\partial^2 u}{\partial x_i^2}+\frac{k_i}{x_i}\frac{\partial u}{\partial x_i}\right) -\lambda^2 u,\qquad \lambda\in\mathbb{R}
\end{equation}
was solved in \cite{ShishKlein}.
GEPD equation with  potential $q(x)$
\begin{equation}\label{PotEPD1}
\frac{\partial^2 u}{\partial t^2}+\frac{\nu}{t}\frac{\partial u}{\partial t}=
\frac{\partial^2 u}{\partial x^2}+\frac{k}{x}\frac{\partial u}{\partial x}+q(x)u
\end{equation}
was studied in  \cite{Sta,volk}.

For the case of (\ref{PotEPD1}) with one variable potential function the  effective spectral parameter
power series (SPPS) method was developed by V.V.Kravchenko and his coathors in \cite{Castillo1,KTS,KTC2017}.
But earlier similar series for deriving transmutations for perturbed Bessel operators were considered in \cite{CFH,FH}.

Results and problems for abstract EPD equation
\begin{equation}\label{AbsEPD}
\frac{\partial^2 u}{\partial t^2}+\frac{\nu}{t}\frac{\partial u}{\partial t}=
Au,\qquad u=u(x,t),\qquad t>0,\qquad x\in\mathbb{R}^n,\qquad
\nu=const,
\end{equation}
where $A$ is a closed densely defined operator in  Hilbert, Banach or Fr\'{e}chet spaces   were summed up and mainly initiated in \cite{CSh,Car1}. Problems in Banach or Hilbert spaces with abstract Bessel operators were studied after that in  many papers, e.g. \cite{GKSh,Glushak0,Glushak1,Glushak2,Glushak3}. But it seems that studies of abstract differential equations in Fr\'{e}chet spaces somehow ceased after initial impulse from \cite{CSh,Car1}, but such problems are very important as for example differential equations on half--spaces or other unbounded spaces need Fr\'{e}chet spaces, and not Hilbert or Banach ones.
The EPD/GEPD--type singular differential equations  appear in different applied problems  as well as  mathematical problems in partial differential equations,  harmonic analysis, generalised translation theory, transmutation theory, numerical analysis and so on, cf. \cite{CSh,Car1,Car2,Car3,Smirnov1,Tersenov,Weinstein0,Rad2,Sit2,LPSh1,LPSh2,ShishKlein,ShSitPul}.

Different equations with Bessel operators  (\ref{EPDWe}--\ref{EPDSP1}) are special cases of perturbed general linear differential equation with nonconstant coefficients
\begin{equation}\label{GenEPD}
\sum\limits_{k=1}^n A_k\left(   \frac{\partial^2 u}{\partial x_k^2}+\frac{\nu_k}{x_k}\frac{\partial u}{\partial x_k}\right) \pm \lambda^2 u=0,
\end{equation}
$$
x_k>0,\qquad A_k=const,\qquad \nu_k=const,\qquad \lambda=const,
$$
corresponding to non--perturbed linear differential equation with constant coefficients
\begin{equation}\label{GenWave}
\sum\limits_{k=1}^n A_k  \frac{\partial^2 v}{\partial x_k^2}\pm \lambda^2 v=0.
\end{equation}

The main topic of this paper is to study a general method based on transmutations --- we call it {\bf integral transforms composition method (ITCM)} --- to derive connection formulas representing solutions of equation (\ref{GenEPD}) $u(x)$  via solutions of the same equation but with different parameters $\nu_k$. And as special case it reduces to connection formulas among solutions of perturbed equation (\ref{GenEPD}) $u(x)$ and solutions of non--pertubed equations $v(x)$ from (\ref{GenWave}) and v.v.
Such relations are called {\bf parameter shift formulas}. Aforesaid formulas arise when classical wave equation is solved by mean values method. The descent parameter in this case is the space dimension. Essentially,
such parameter shift formulas define transmutation operators which are responsible for connection formulas among solutions of perturbed and non--perturbed equations.
To construct the necessary transmutation operator we use the integral transforms composition method (ITCM)  introduced and thoroughly developed in \cite{Sit0,Sit01,Sit1,Sit2,Sit3,Sit4}. The essence of this method is to construct the necessary transmutation operator and corresponding connection formulas among solutions of perturbed and non--perturbed equations as a composition of classical integral transforms with properly chosen
weighted functions.

We note that other possible generalizations of considered equations are equations with fractional powers of the Bessel operator considered in \cite{McBArt,Ida,Dim,Kir1,Sita1,Sita2,SitSh1,SitSh2}. In fractional differential equations theory the so called "principle of subordination"\, was proposed (cf. \cite{Jan,EIK,Bajlekova0,Bajlekova1}). In the cited literature principle of subordination is reduced to  formulas  relating the solutions to equations of various fractional orders. A special case of subordination principle are formulas connected solutions of fractional differential equations to solutions of integer order equations. Such formulas are also in fact parameter shift formulas, in which the parameter is the order of fractional DE. So a popular "principle of subordination"\, may be considered as an example of parameter shift formulas, and consequently is in close connection with transmutation theory and ITCM developed in this paper.

Note that we specially restrict ourself to linear problems, but of course nonlinear problems are also very important, cf. \cite{Rad3,Rad4} for further references.

We also introduce and use convenient hybrid terminology of I.A.Kipriyanov (\textbf{\textit{B--elliptic, B--parabolic, B--hyperbolic differential equations}}) for differential equations with Bessel operators and of R.Carroll (\textbf{\textit{elliptic, parabolic, hyperbolic transmutations}}). As a result we use terms \textbf{\textit{B--elliptic, B--parabolic and B--hyperbolic transmutations}} for those ones which intertwine Bessel operators to first (B--parabolic) and $\pm$ second derivatives (B--elliptic and B--hyperbolic).

\section{Basic definitions}

Here we provide definitions and brief information on the special functions, classes of functions and integral transforms.

The hypergeometric function is defined for $|x| < 1$ by the power series
\begin{equation}\label{2F1}
 {}_{2}F_{1}(a,b;c;x)=\sum _{n=0}^{\infty }{\frac {(a)_{n}(b)_{n}}{(c)_{n}}}{\frac {x^{n}}{n!}}.
\end{equation}

 The Bessel function of the first kind of order $\gamma$ is  defined  by its series expansion around $x=0$:
$$
J_{\gamma }(x)=\sum _{m=0}^{\infty }{\frac {(-1)^{m}}{m!\,\Gamma (m+\gamma +1)}}{\left({\frac {x}{2}}\right)}^{2m+\gamma}
$$
is (see \cite{Watson}).

Let $S$ be the space of rapidly decreasing functions on $(0,\infty)$
$$
S=\left\{f\in C^\infty(0,\infty):\sup _{x\in (0,\infty)}\left|x^{\alpha }D^{\beta }f(x)\right|<\infty \quad \forall \alpha ,\beta \in {Z} _{+}\right\}.
$$

The Hankel transform of order $\nu$ of a function $f\in S$ is given by:
\begin{equation}\label{SK11}
H_\nu [f](x)=\widehat{f}(x)=\int\limits_{0}^{\infty }{j}_{\frac{\nu-1}{2}}(xt) f(t) t^\nu dt,\qquad
\nu\neq -1,-3,-5,...
\end{equation}
where
\begin{equation}\label{HSM}
j_\gamma(t){=}2^\gamma\,\Gamma(\gamma+1)\frac{J_\gamma(t)}{t^\gamma}.
\end{equation}
  The Hankel transform defined in this way is also its own inverse up to a constant:
$$H_\nu^{-1}[\widehat{f}](x)=f(x)=\frac{2^{1-\nu}}{\Gamma^2\left(\frac{\nu+1}{2}\right)}
\int\limits_{0}^{\infty} {j}_{\frac{\nu-1}{2}} (x\xi)\,
\widehat{f}(\xi)\xi^\nu\,d\xi.
$$
The Hankel transform and its inverse work for all functions in $L^2(0, \infty)$.




\section{What is ITCM and how to use it?}

In transmutation theory explicit operators were derived based on different ideas and methods, often not connecting altogether. So there is an urgent need in transmutation theory to develop a general method for obtaining known and new classes of transmutations.

In this section we give such general method for constructing  transmutation operators.
We call this method \textbf{\textit{integral transform composition method}} or shortly ITCM.
The method is based on the representation of transmutation operators as  compositions of basic  integral transforms.
The  \textbf{\textit{integral transform composition method}} (ITCM) gives the algorithm not only for constructing new transmutation operators, but also for all now explicitly known classes of transmutations, including Poisson, Sonine, Vekua--Erdelyi--Lowndes, Buschman-Erdelyi, Sonin--Katrakhov and Poisson--Katrakhov ones, cf.
\cite{CSh,Car1,Car2,Car3,Sit0,Sit01,Sit1,Sit2,Sit3,Sit4}
 as well as the classes of elliptic, hyperbolic and parabolic transmutation operators introduced by R. Carroll \cite{Car1,Car2,Car3}.

The formal algorithm of ITCM is the next. Let us take as input a pair of arbitrary operators $A,B$, and also connecting with them generalized Fourier transforms $F_A, F_B$, which are invertible and act by the formulas
\begin{equation}
\label{4301}
F_A A =g(t) F_A,\ \ \  F_B B= g(t) F_B,
\end{equation}
$t$ is a dual variable, $g$ is an arbitrary function with suitable properties.
It is often convenient to choose $g(t)=-t^2$ or $g(t)=-t^\alpha$, $\alpha\in\mathbb{R}$.

\textbf{\textit{Then the essence of ITCM is to  obtain formally a pair of  transmutation operators $P$ and $S$ as the method output by the next formulas:
\begin{equation}\label{Comp}
S=F^{-1}_B \frac{1}{w(t)} F_A,\qquad P=F^{-1}_A w(t) F_B
\end{equation}
with  arbitrary function $w(t)$.
When $P$ and $S$  are
transmutation operators intertwining $A$ and $B$}}:
\begin{equation}\label{Inter}
SA=BS,\qquad PB=AP.
\end{equation}
A formal checking of (\ref{Inter}) can be obtained by direct substitution.
The main difficulty is the calculation of  compositions (\ref{Comp})
in an explicit integral form, as well as the choice of   domains of operators $P$ and $S$.

Let us list main advantages of Integral Transform Composition Method (ITCM).

\begin{itemize}
\item{Simplicity --- many classes of transmutations are obtained by explicit formulas from elementary basic blocks, which are classical integral transforms.}
\item{ITCM gives by a unified approach all previously explicitly known classes of transmutations.}
\item{ITCM gives by a unified approach many new classes of transmutations for different  operators.}
\item{ITCM gives  a unified approach to obtain both direct and inverse transmutations in the same composition form.}
\item{ITCM directly leads to estimates of norms of direct and inverse transmutations using known norm estimates for classical integral transforms on different functional spaces.}
\item{ITCM directly leads to connection formulas for solutions to perturbed and unperturbed differential equations.}
\end{itemize}

Some obstacle to apply ITCM is the next one: we know acting of classical integral transforms usually on standard  spaces like $L_2, L_p, C^k$, variable exponent Lebesgue spaces \cite{Rad1} and so on. But for application of transmutations to differential equations we usually need some more conditions hold, say at zero or at infinity. For these problems we may first construct a transmutation by ITCM  and then expand it to the needed functional classes.

Let us stress that formulas of the type \eqref{Comp} of course are not new for integral transforms and its applications to differential equations. \textbf{\textit{But ITCM  is new when applied to transmutation theory!}} In other fields of integral transforms and connected differential equations theory compositions \eqref{Comp} for the choice of classical Fourier transform leads to famous pseudo--differential operators with symbol function $w(t)$. For  the choice of the classical Fourier transform and the function
 $w(t)=(\pm it)^{-s}$ we get fractional integrals on the whole real axis, for $w(t)=|x|^{-s}$ we get M.Riesz potential, for
 $w(t)=(1+t^2)^{-s}$ in formulas \eqref{Comp} we get Bessel potential and for $w(t)=(1\pm it)^{-s}$ --- modified Bessel potentials \cite{SKM}.

The next choice for ITCM  algorithm:
\begin{equation}\label{trans}
A=B=B_\nu, F_A=F_B=H_\nu, g(t)=-t^2, w(t)=j_\nu(st)
\end{equation}
leads to generalized translation operators of Delsart \cite{Levitan1,Levitan2,Levitan3}, for this case we have to choose in ITCM algorithm defined by \eqref{4301}--\eqref{Comp} the above values \eqref{trans} in which $B_\nu$ is the Bessel operator \eqref{Bessel}, $H_\nu$ is the Hankel transform \eqref{SK11}, $j_\nu$ is the normalized (or "small") Bessel function \eqref{HSM}. In the same manner other families of operators commuting with a given one may be obtained by ITCM for the choice $A=B, F_A=F_B$ with arbitrary functions $g(t), w(t)$ (generalized translation commutes with the Bessel operator).
In case of the choice of differential operator $A$ as quantum oscillator and connected integral transform $F_A$ as fractional or quadratic Fourier transform \cite{OZK} we may obtain by ITCM transmutations also for this case \cite{Sit2}. It is  possible to apply ITCM instead of classical approaches for obtaining fractional powers of Bessel operators \cite{Sita1,Sita2,Sit2,SitSh1,SitSh2}.

Direct applications of ITCM to multidimensional differential operators are obvious, in this case $t$ is a vector and $g(t),w(t)$ are vector functions in \eqref{4301}--\eqref{Comp}. Unfortunately for this case we know and may derive some new explicit transmutations just for simple special cases. But among them are well--known and interesting classes of potentials. In case of using ITCM by \eqref{4301}--\eqref{Comp} with Fourier transform and  $w(t)$ --- positive definite quadratic form we come to elliptic M.Riesz potentials \cite{Riesz,SKM}; with  $w(t)$ --- indefinite quadratic form we come to hyperbolic M.Riesz potentials \cite{Riesz,SKM,Nogin}; with $w(x,t)=(|x|^2-it)^{-\alpha/2}$  we come to parabolic potentials \cite{SKM}. In case of using ITCM by \eqref{4301}--\eqref{Comp} with Hankel transform and  $w(t)$ --- quadratic form we come to elliptic M.Riesz B--potentials \cite{Lya1,Gul1} or hyperbolic M.Riesz B--potentials \cite{Shi1}. For all above mentioned potentials we need to use distribution theory and consider for ITCM convolutions of distributions, for inversion of such potentials we need some cutting and approximation procedures, cf. \cite{Nogin,Shi1}. For this class of problems it is appropriate to use Schwartz or/and Lizorkin spaces for probe functions and dual spaces for distributions.

So we may conclude that the method we consider in the paper for obtaining transmutations --- ITCM is effective, it is connected to many known methods and problems, it gives all known classes of explicit transmutations and works as a tool to construct new classes of transmutations. Application of ITCM needs the next three steps.

\begin{itemize}
\item{Step 1. For a given pair of operators $A,B$ and connected generalized Fourier transforms $F_A, F_B$ define and calculate a pair of transmutations $P,S$ by basic formulas \eqref{4301}--\eqref{Comp}.}
\item{Step 2. Derive exact conditions and find classes of functions for which  transmutations obtained by step 1 satisfy proper intertwining properties.}
\item{Step 3. Apply now correctly defined transmutations by steps 1 and 2 on proper classes of functions to deriving connection formulas for solutions of differential equations.}
\end{itemize}

Due to this plan the next part of the paper is organized as follows. First we illustrate step 1 of the above plan and apply ITCM for obtaining some new and known transmutations. For step 2 we prove a general theorem for the case of Bessel operators, it is enough to solve problems to complete strict definition of transmutations. And after that we give an example to illustrate step 3 of applying obtained by ITCM  transmutations to derive formulas for solutions of a model differential equation.

\section{Application of ITCM to index shift $B$--hyperbolic   transmutations}

In this section  we apply ITCM to obtain integral representations for  index shift $B$--hyperbolic   transmutations. It corresponds to step 1 of the above plan for ITCM algorithm.

Let us look for the operator $T$ transmuting the operator $B_\nu$ defined by (\ref{Bessel}) into the same operator but with another parameter $B_\mu$.
To find such a transmutation  we use ITCM with Hankel transform. Applying ITCM we obtain an interesting and important family of transmutations, including index shift $B$--hyperbolic   transmutations, "descent" operators, classical Sonine and Poisson--type transmutations, explicit integral representations for fractional powers of the Bessel operator, generalized translations of Delsart and others.

So we are looking for an operator $T_{\nu, \, \mu}^{(\varphi)}$ such that
\begin{equation}\label{449}
T_{\nu, \, \mu}^{(\varphi)} B_{\nu} = B_{\mu} T_{\nu, \, \mu}^{(\varphi)}
\end{equation}
in the factorised due to ITCM form
\begin{equation}\label{4410}
T_{\nu, \, \mu}^{(\varphi)} = H_{\mu}^{-1} \biggl( \varphi(t) H_{\nu}\biggr),
\end{equation}
where $H_\nu$ is a Hankel transform \eqref{SK11}.
Assuming $\varphi (t) = Ct^{\alpha}$, $C\in\mathbb{R}$ does not depend on $t$ and $T^{(\varphi)}_{\nu,\, \mu}=T^{(\alpha)}_{\nu,\, \mu}$ we can derive the  following theorem.

\begin{theorem} Let $f\in L^2(0, \infty)$,
	$$
	{\rm Re}\,(\alpha+\mu+1)>0;\qquad {\rm Re}\,\left(\alpha+\frac{\mu-\nu}{2} \right)<0.
	$$
	Then for transmutation operator $T^{(\alpha)}_{\nu,\, \mu}$ obtained by ITCM
		and such that
		$$
		T^{(\alpha)}_{\nu,\, \mu}  B_{\nu} = B_{\mu} T^{(\alpha)}_{\nu,\, \mu}
		$$
	the next integral representation is true
		$$
		\left( T^{(\alpha)}_{\nu,\, \mu} f\right)(x) =C\,\cdot\frac{2^{\alpha+3}\Gamma\left(\frac{\alpha+\mu+1}{2}\right)}{\Gamma\left(\frac{\mu+1}{2}\right)}\times
		$$
		$$
		\times\left[\frac{x^{-1-\mu-\alpha}}
		{\Gamma\left(-\frac{\alpha}{2}\right)} \,\int\limits_{0}^{x}f(y)
		{_2F_1}\left( \frac{\alpha+\mu+1}{2}, \frac{\alpha}{2}+1; \frac{\nu+1}{2}; \frac{y^2}{x^2}\right)y^\nu dy+\right.
		$$
		$$
		+\frac{\Gamma\left(\frac{\nu+1}{2}\right)
		}{\Gamma\left(\frac{\mu+1}{2}\right)\Gamma\left(\frac{\nu-\mu-\alpha}{2}\right)}
		\,\int\limits_{x}^{\infty }f(y)\times
		$$	
		\begin{equation}\label{Theo1}
		\left.
		\times
		{_2F_1}\left( \frac{\alpha+\mu+1}{2}, \frac{\alpha+\mu-\nu}{2}+1;  \frac{\mu+1}{2}; \frac{x^2}{y^2}\right)y^{\nu-\mu-\alpha-1}dy\right].
		\end{equation}
		where ${_2F_1}$ is the Gauss hypergeometric function.
	\end{theorem}

	\begin{proof} We have
		$$
		\left( T^{(\alpha)}_{\nu,\, \mu} f\right) (x) =C\,\cdot	H_{\mu}^{-1} \left[t^{\alpha} H_{\nu}[f](t)\right](x)=
		$$
		$$
		=C\,\cdot\frac{2^{1-\mu}}{\Gamma^2\left(\frac{\mu+1}{2}\right)}
		\int\limits_{0}^{\infty} {j}_{\frac{\mu-1}{2}} (xt)\,
		t^{\mu+\alpha}\,dt \int\limits_{0}^{\infty }{j}_{\frac{\nu-1}{2}}(ty) f(y) y^\nu dy=
		$$	
		$$
		=C\,\cdot\frac{2^{\frac{\nu-\mu}{2}+2}\Gamma\left(\frac{\nu+1}{2}\right)}{\Gamma\left(\frac{\mu+1}{2}\right)} \int\limits_{0}^{\infty }(xt)^{\frac{1-\mu}{2}}J_{\frac{\mu-1}{2} }(xt) t^{\mu+\alpha}\, dt
		\int\limits_{0}^{\infty }(ty)^{\frac{1-\nu}{2}}J_{\frac{\nu-1}{2}}(ty) f(y)y^\nu dy=
		$$
		$$
		=C\,\cdot\frac{2^{\frac{\nu-\mu}{2}+2}\Gamma\left(\frac{\nu+1}{2}\right)}{\Gamma\left(\frac{\mu+1}{2}\right)} \,x^{\frac{1-\mu}{2}}\int\limits_{0}^{\infty }y^{\frac{\nu+1}{2}}f(y)dy
		\int\limits_{0}^{\infty } t^{\alpha+1+\frac{\mu-\nu}{2}} J_{\frac{\mu-1}{2} }(xt)J_{\frac{\nu-1}{2}}(ty)dt =
		$$
		$$
		=C\,\cdot\frac{2^{\frac{\nu-\mu}{2}+2}\Gamma\left(\frac{\nu+1}{2}\right)}{\Gamma\left(\frac{\mu+1}{2}\right)} \,x^{\frac{1-\mu}{2}}\int\limits_{0}^{x}y^{\frac{\nu+1}{2}}f(y)dy
		\int\limits_{0}^{\infty } t^{\alpha+1+\frac{\mu-\nu}{2}} J_{\frac{\mu-1}{2} }(xt)J_{\frac{\nu-1}{2}}(ty)dt+
		$$
		$$
		+C\,\cdot\frac{2^{\frac{\nu-\mu}{2}+2}\Gamma\left(\frac{\nu+1}{2}\right)}{\Gamma\left(\frac{\mu+1}{2}\right)} \,x^{\frac{1-\mu}{2}}\int\limits_{x}^{\infty }y^{\frac{\nu+1}{2}}f(y)dy
		\int\limits_{0}^{\infty } t^{\alpha+1+\frac{\mu-\nu}{2}} J_{\frac{\mu-1}{2} }(xt)J_{\frac{\nu-1}{2}}(ty)dt.
		$$

		Using formula 2.12.31.1 from \cite{IR2} p. 209 of the form
		$$
		\int\limits_{0}^{\infty } t^{\beta-1}\, J_{\rho }(xt)J_{\gamma }(yt)\, dt=
		$$		
		$$
		=\left\{%
		\begin{array}{ll}
		$$2^{\beta-1}x^{-\gamma-\beta}y^\gamma
		\frac{\Gamma\left(\frac{\gamma+\rho+\beta}{2}\right)}{\Gamma(\gamma+1)
			\Gamma\left(\frac{\rho-\gamma-\beta}{2}+1\right)}
		{_2F_1}\left( \frac{\gamma+\rho+\beta}{2}, \frac{\gamma-\rho+\beta}{2}; \gamma+1; \frac{y^2}{x^2}\right) $$, & \hbox{$0<y<x$;} \\
		$$2^{\beta-1}x^{\rho}y^{-\rho-\beta}
		\frac{\Gamma\left(\frac{\gamma+\rho+\beta}{2}\right)}{\Gamma(\rho+1)
			\Gamma\left(\frac{\gamma-\rho-\beta}{2}+1\right)}
		{_2F_1}\left( \frac{\gamma+\rho+\beta}{2}, \frac{\beta+\rho-\gamma}{2}; \rho+1; \frac{x^2}{y^2}\right)
		$$, & \hbox{$0<x<y,$} \\
		\end{array}%
		\right.
		$$
		$$
		x,y, {\rm Re}\,(\beta+\rho+\gamma)>0;\,\, {\rm Re}\,\beta<2
		$$
		and putting $\beta=\alpha+\frac{\mu-\nu}{2}+2$, $\rho=\frac{\mu-1}{2}$, $\gamma=\frac{\nu-1}{2}$ we obtain (\ref{Theo1}).
		$$
		\int\limits_{0}^{\infty } t^{\alpha+1+\frac{\mu-\nu}{2}} J_{\frac{\mu-1}{2} }(xt)J_{\frac{\nu-1}{2}}(ty)dt=
		$$
		$$
		=\left\{%
		\begin{array}{ll}
		$$\frac{2^{\alpha+1+\frac{\mu-\nu}{2}}y^{\frac{\nu-1}{2}}}{x^{\alpha+2-\frac{1-\mu}{2}}}
		\frac{\Gamma\left(\frac{\alpha+\mu+1}{2}\right)}{\Gamma\left( \frac{\nu+1}{2}\right)
			\Gamma\left(-\frac{\alpha}{2}\right)}
		{_2F_1}\left( \frac{\alpha+\mu+1}{2}, \frac{\alpha}{2}+1; \frac{\nu+1}{2}; \frac{y^2}{x^2}\right) $$, & \hbox{$0<y<x$;} \\
		$$\frac{2^{\alpha+1+\frac{\mu-\nu}{2}}x^{\frac{\mu-1}{2}}}{y^{\mu+\alpha-\frac{\nu-3}{2}}}
		\frac{\Gamma\left(\frac{\alpha+\mu+1}{2}\right)}{\Gamma\left(\frac{\mu+1}{2}\right)
			\Gamma\left(\frac{\nu-\mu-\alpha}{2}\right)}
		{_2F_1}\left( \frac{\alpha+\mu+1}{2}, \frac{\alpha+\mu-\nu}{2}+1;  \frac{\mu+1}{2}; \frac{x^2}{y^2}\right)
		$$, & \hbox{$0<x<y,$} \\
		\end{array}%
		\right.
		$$
		$$
		{\rm Re}\,(\alpha+\mu+1)>0;\qquad {\rm Re}\,\left(\alpha+\frac{\mu-\nu}{2} \right)<0
		$$
		and
		$$
		\left( T^{(\alpha)}_{\nu,\, \mu} f\right)(x) =
		$$
		$$
		=C\,\cdot\frac{2^{\alpha+3}\Gamma\left(\frac{\alpha+\mu+1}{2}\right)}{\Gamma\left(-\frac{\alpha}{2}\right)\Gamma\left(\frac{\mu+1}{2}\right)} \,x^{-1-\mu-\alpha}\int\limits_{0}^{x}f(y)
		{_2F_1}\left( \frac{\alpha+\mu+1}{2}, \frac{\alpha}{2}+1; \frac{\nu+1}{2}; \frac{y^2}{x^2}\right)y^\nu dy+
		$$
		$$
		+C\,\cdot\frac{2^{\alpha+3}\Gamma\left(\frac{\nu+1}{2}\right)
			\Gamma\left(\frac{\alpha+\mu+1}{2}\right)}{\Gamma^2\left(\frac{\mu+1}{2}\right)\Gamma\left(\frac{\nu-\mu-\alpha}{2}\right)}
		\,\int\limits_{x}^{\infty }f(y)
		{_2F_1}\left( \frac{\alpha+\mu+1}{2}, \frac{\alpha+\mu-\nu}{2}+1;  \frac{\mu+1}{2}; \frac{x^2}{y^2}\right)\times$$
		$$\times y^{\nu-\mu-\alpha-1}dy.
		$$	
		This completes the proof.
	\end{proof}


Constant $C$ in (\ref{Theo1}) should be chosen based on convenience.
Very often it is reasonable to choose this constant so that $T^{(\alpha)}_{\nu,\, \mu}\, 1=1 $.

 	Using formula
\begin{equation}
\label{Hyppo}
	_2F_1(a,b;b;z)=(1-z)^{-a}
\end{equation}
	we  give several useful transmutation operators that are special cases of
	operator (\ref{Theo1}). In section \ref{APPL} we will use these operators to find the solutions to the perturbed wave equations.

\begin{corollary}
	Let  $f\in L^2(0, \infty)$, $\alpha=-\mu$; $\nu=0$. In this case for $\mu>0$ we obtain
	the operator
	\begin{equation}
	\label{Poisson1}
	\left( T^{(-\mu)}_{0,\, \mu} f\right)(x)=\frac{2\Gamma\left(\frac{\mu+1}{2}\right)}{\sqrt{\pi}
		\Gamma\left(\frac{\mu}{2}\right)}\,x^{1-\mu}\,\int\limits_{0}^{x}f(y)
	(x^2- y^2)^{\frac{\mu}{2}-1}dy,
	\end{equation}
	such that
	\begin{equation}\label{4499}
	T_{0, \, \mu}^{(-\mu)} D^2 = B_{\mu} T_{0, \, \mu}^{(-\mu)}
	\end{equation}
	and $T^{(-\mu)}_{0,\, \mu}\,1=1$,
\end{corollary}
\begin{proof} We have
		$$
		\left( T^{(-\mu)}_{0,\, \mu} f\right)(x)
		=C\,\cdot\frac{2^{3-\mu}\sqrt{\pi}}{x\Gamma\left(\frac{\mu}{2}\right)
			\Gamma\left(\frac{\mu+1}{2}\right)}
	\int\limits_{0}^{x}f(y)
		{_2F_1}\left( \frac{1}{2}, 1-\frac{\mu}{2}; \frac{1}{2}; \frac{y^2}{x^2}\right) dy.
		$$
	 Using formula (\ref{Hyppo})
	we get
$$
{_2F_1}\left( \frac{1}{2}, 1-\frac{\mu}{2}; \frac{1}{2}; \frac{y^2}{x^2}\right)=\left(1- \frac{y^2}{x^2}\right)^{\frac{\mu}{2}-1}=x^{2-\mu}(x^2-y^2)^{\frac{\mu}{2}-1}
$$	
and
$$
\left( T^{(-\mu)}_{0,\, \mu} f\right)(x)
=C\,\cdot\frac{x^{1-\mu}2^{3-\mu}\sqrt{\pi}}{\Gamma\left(\frac{\mu}{2}\right)
	\Gamma\left(\frac{\mu+1}{2}\right)}
\int\limits_{0}^{x}f(y)
(x^2-y^2)^{\frac{\mu}{2}-1}dy.
$$
It is easy to see that
$$
x^{1-\mu}\int\limits_{0}^{x}
(x^2-y^2)^{\frac{\mu}{2}-1}dy=\{y=xz\}=\int\limits_{0}^{1}
(1-z^2)^{\frac{\mu}{2}-1}dz=
$$
$$
=\{z^2=t\}=\frac{1}{2}\int\limits_{0}^{1}
(1-t)^{\frac{\mu}{2}-1}t^{-\frac{1}{2}}dt
= \frac{\sqrt{\pi}
	\Gamma\left(\frac{\mu}{2}\right)}{2\Gamma\left(\frac{\mu+1}{2}\right)}
$$
and taking $C=\frac{\Gamma^2\left(\frac{\mu+1}{2}\right)}{2^{2-\mu}\pi}$ we get
$T^{(-\mu)}_{0,\, \mu}1 =1$.
And this completes the proof.
 \end{proof}

The operator (\ref{Poisson1}) is the well--known {\it Poisson operator} (see \cite{Levitan1}). We will use conventional  symbol $\mathcal{P}^\mu_x$ for it:
	\begin{equation}
	\label{Poisson}
\mathcal{P}^\mu_x f(x)=C(\mu)x^{1-\mu}\,\int\limits_{0}^{x}f(y)
	(x^2- y^2)^{\frac{\mu}{2}-1}dy,
	\end{equation}
	$$
	\mathcal{P}^\mu_x 1=1,  \qquad C(\mu)=\frac{2\Gamma\left(\frac{\mu+1}{2}\right)}{\sqrt{\pi}
		\Gamma\left(\frac{\mu}{2}\right)}.
	$$

{\bf Remark}. It is easy to see that
if  $u=u(x,t)$, $x,t\in \mathbb{R}$ and
$$
u(x,0)=f(x),\qquad u_t(x,0)=0
$$
then
\begin{equation}\label{InC}
	\mathcal{P}^\mu_t u(x,t)|_{t=0}=f(x),\qquad \frac{\partial}{\partial t} \mathcal{P}^\mu_t u(x,t)\biggr|_{t=0}=0.
\end{equation}
Indeed, we have
$$
	\mathcal{P}^\mu_t u(x,t)|_{t=0}=C(\mu)t^{1-\mu}\,\int\limits_{0}^{t}u(x,y)
(t^2- y^2)^{\frac{\mu}{2}-1}dy\biggr|_{t=0}=
$$
$$
=C(\mu)\,\int\limits_{0}^{1}u(x,ty)|_{t=0}
(1- y^2)^{\frac{\mu}{2}-1}dy=f(x)
$$
and
$$
 \frac{\partial}{\partial t} \mathcal{P}^\mu_t u(x,t)\biggr|_{t=0}
=C(\mu)\,\int\limits_{0}^{1}u_t(x,ty)|_{t=0}
(1- y^2)^{\frac{\mu}{2}-1}dy=0.
$$

\begin{corollary} 	For  $f\in L^2(0, \infty)$, $\alpha{=}\nu{-}\mu$; $-1{<} {\rm Re}\, \nu {<} {\rm Re}\, \mu$ we obtain the first "descent"\, operator
		\begin{equation}\label{OPDBess}
		\left( T^{(\nu-\mu)}_{\nu,\, \mu} f\right)(x) =\frac{2\,\Gamma\left(\frac{\mu+1}{2}\right)}{\Gamma\left(\frac{\mu{-}\nu}{2}\right)
			\Gamma\left(\frac{\nu+1}{2}\right)}\,x^{1-\mu} \int\limits_{0}^{x}f(y)
		(x^2-y^2)^{\frac{\mu-\nu}{2}-1}y^\nu dy.
		\end{equation}
	such that
		$$
	T^{(\nu-\mu)}_{\nu,\, \mu}  B_{\nu} = B_{\mu} T^{(\nu-\mu)}_{\nu,\, \mu}
		$$
	 and
$$T^{(\nu-\mu)}_{\nu,\, \mu}1 =1.$$	
\end{corollary}
\begin{proof} Substituting the value $\alpha=\nu-\mu$  into  (\ref{Theo1})
	 we can write
$$
\left( T^{(\nu-\mu)}_{\nu,\, \mu} f\right)(x) =C\,\cdot\frac{2^{\nu-\mu+3}\Gamma\left(\frac{\nu+1}{2}\right)}
{\Gamma\left(\frac{\mu-\nu}{2}\right)\Gamma\left(\frac{\mu+1}{2}\right)} \,x^{-1-\nu} \times
$$
$$
\times\int\limits_{0}^{x}f(y)
{_2F_1}\left( \frac{\nu+1}{2}, \frac{\nu-\mu}{2}+1; \frac{\nu+1}{2}; \frac{y^2}{x^2}\right)y^\nu dy.
$$
Taking  into account the identity (\ref{Hyppo}) for a hypergeometric function the last equality reduces to
$x^{\nu-\mu+2}(x^2-y^2)^{\frac{\mu-\nu}{2}-1}$ and the operator $ T^{(\nu-\mu)}_{\nu,\, \mu}$ is written in the form
$$
\left( T^{(\nu-\mu)}_{\nu,\, \mu} f\right)(x) =C\,\cdot\frac{2^{\nu-\mu+3}\Gamma\left(\frac{\nu+1}{2}\right)}
{\Gamma\left(\frac{\mu-\nu}{2}\right)\Gamma\left(\frac{\mu+1}{2}\right)}\,x^{1-\mu} \int\limits_{0}^{x}f(y)
(x^2-y^2)^{\frac{\mu-\nu}{2}-1}y^\nu dy.
$$
Clearly that
$$
x^{1-\mu} \int\limits_{0}^{x}
(x^2-y^2)^{\frac{\mu-\nu}{2}-1}y^\nu dy=\{y=xz\}=\int\limits_{0}^{1}
(1-z^2)^{\frac{\mu-\nu}{2}-1}z^\nu dz=
$$
$$
=\{z^2=t\}=\frac{1}{2}\int\limits_{0}^{1}
(1-t)^{\frac{\mu-\nu}{2}-1}t^{\frac{\nu-1}{2}}dt
= \frac{\Gamma\left(\frac{\mu-\nu}{2}\right)\Gamma\left(\frac{\nu+1}{2}\right)}{2\Gamma\left(\frac{\mu+1}{2}\right)}
$$
and taking $C=\frac{2^{\mu-\nu-2}\Gamma^2\left(\frac{\mu+1}{2}\right)}{\Gamma^2\left(\frac{\nu+1}{2}\right)}$ we get
$T^{(\nu-\mu)}_{\nu,\, \mu}1 =1$.
It completes the proof.

\end{proof}

\begin{corollary}
	Let  $f\in L^2(0, \infty)$, $\alpha=0$; $-1{<}{\rm Re}\, \mu {<} {\rm Re}\, \nu$. In this case we obtain
	the second "descent"\, operator:
	\begin{equation}
	\label{desent}
	\left( T^{(0)}_{\nu,\, \mu} f\right)(x) =\frac{2\Gamma\left(\nu-\mu\right)}
	{\Gamma^2\left(\frac{\nu-\mu}{2}\right)}
	\,\int\limits_{x}^{\infty }f(y)
	(y^2-x^2)^{\frac{\nu-\mu}{2}-1}ydy.
	\end{equation}
\end{corollary}
\begin{proof} We have
			$$
			\left( T^{(0)}_{\nu,\, \mu} f\right)(x) =$$
			$$=
		C\,\cdot\frac{2^{3}\Gamma\left(\frac{\nu+1}{2}\right)
			}{\Gamma\left(\frac{\mu+1}{2}\right)\Gamma\left(\frac{\nu-\mu}{2}\right)}
			\,\int\limits_{x}^{\infty }f(y)
			{_2F_1}\left( \frac{\mu+1}{2}, \frac{\mu-\nu}{2}+1;  \frac{\mu+1}{2}; \frac{x^2}{y^2}\right) y^{\nu-\mu-1}dy.
			$$
		Using formula (\ref{Hyppo})
	we get
		$$
		{_2F_1}\left( \frac{\mu+1}{2}, \frac{\mu-\nu}{2}+1;  \frac{\mu+1}{2}; \frac{x^2}{y^2}\right)
		=\left( 1-\frac{x^2}{y^2}\right)^{\frac{\nu-\mu}{2}-1}=y^{2+\mu-\nu}(y^2-x^2)^{\frac{\nu-\mu}{2}-1}
	$$
and	
$$
\left( T^{(0)}_{\nu,\, \mu} f\right)(x) =
C\,\cdot\frac{2^{3}\Gamma\left(\frac{\nu+1}{2}\right)
}{\Gamma\left(\frac{\mu+1}{2}\right)\Gamma\left(\frac{\nu-\mu}{2}\right)}
\,\int\limits_{x}^{\infty }f(y)
(y^2-x^2)^{\frac{\nu-\mu}{2}-1}ydy.
$$

It is obvious that
$$
\int\limits_{x}^{\infty }(y^2-x^2)^{\frac{\nu-\mu}{2}-1}ydy=\left\{y=\frac{x}{z}\right\}=
$$
$$
=x^{\nu-\mu}\int\limits_{0}^{1}(1-z^2)^{\frac{\nu-\mu}{2}-1}z^{\mu-\nu-1}dz=\{z^2=t\}=
$$
$$
=\frac{x^{\nu-\mu}}{2}\int\limits_{0}^{1}(1-t)^{\frac{\nu-\mu}{2}-1}t^{\frac{\nu-\mu}{2}-1}dt=
\frac{x^{\nu-\mu}\Gamma^2\left(\frac{\nu-\mu}{2}\right)}{2\Gamma(\nu-\mu) }.
$$
Therefore, for  $C=\frac{\Gamma\left(\frac{\mu+1}{2}\right)\Gamma(\nu-\mu)}{4\Gamma\left(\frac{\nu+1}{2}\right)\Gamma\left(\frac{\nu-\mu}{2}\right)}
$ we get
$T^{(\nu-\mu)}_{\nu,\, \mu}1 =x^{\nu-\mu}$.
It completes the proof.
\end{proof}

In \cite{Sit3} the formula (\ref{desent}) was obtained as a particular case of Buschman--Erdelyi operator of the third kind but with different constant:
\begin{equation}
\label{desent1}
\left( T^{(0)}_{\nu,\, \mu} f\right)(x) =\frac{2^{1-\frac{\nu-\mu}{2}}
}{\Gamma\left(\frac{\nu-\mu}{2}\right)}
\,\int\limits_{x}^{\infty }f(y)y
\left(y^2- x^2\right)^{\frac{\nu-\mu}{2}-1}dy.
\end{equation}
As might be seen in the form (\ref{desent}) as well as (\ref{desent1})  the operator $T^{(0)}_{\nu,\, \mu}$ does not depend on the values $\nu$ and $\mu$ but only on the difference between $\nu$ and $\mu$.

\begin{corollary} Let  $f\in L^2(0, \infty)$, $
	{\rm Re}\,(\alpha+\nu+1)>0;\,\, {\rm Re}\,\alpha<0$. If we take $\mu=\nu$ in (\ref{Theo1})
 we obtain the operator
	$$
	\left( T^{(\alpha)}_{\nu,\, \nu} f\right)(x) =\frac{2^{\alpha+3}\Gamma\left(\frac{\alpha+\nu+1}{2}\right)}{\Gamma\left(-\frac{\alpha}{2}\right)\Gamma\left(\frac{\nu+1}{2}\right)}\times
	$$
	$$
	\times\left[x^{-1-\nu-\alpha} \,\int\limits_{0}^{x}f(y)
	{_2F_1}\left( \frac{\alpha+\nu+1}{2}, \frac{\alpha}{2}+1; \frac{\nu+1}{2}; \frac{y^2}{x^2}\right)y^\nu dy+\right.
	$$
	\begin{equation}\label{TheoFr}
	\left.	+
	\,\int\limits_{x}^{\infty }f(y)		{_2F_1}\left( \frac{\alpha+\nu+1}{2}, \frac{\alpha}{2}+1;  \frac{\nu+1}{2}; \frac{x^2}{y^2}\right)y^{-\alpha-1}dy\right]
	\end{equation}
	which is an explicit integral representation of the negative fractional power $\alpha$ of the Bessel operator: $B^\alpha_\nu$.
\end{corollary}

So it is possible and easy to obtain fractional powers of the Bessel operator by ITCM. For different approaches to fractional powers of the Bessel operator and its explicit integral representations cf. \cite{McBArt,Ida,Dim,Kir1,Sita1,Sita2,Sit2,SitSh1,SitSh2}.


\begin{theorem} If  apply ITCM with $\varphi (t) = j_{\frac{\nu-1}{2}}(zt)$ in (\ref{4410}) and with $\mu=\nu$ then the operator
		 $$
	\left( T^{(\varphi)}_{\nu,\, \nu} f\right) (x) =\,^\nu T_x^zf(x)=H_{\nu}^{-1} \left[ j_{\frac{\nu-1}{2}}(z t) H_{\nu}[f](t)\right](x)=
		 $$
	\begin{equation}
	\label{Shift}
		=\frac{2^\nu\Gamma\left(\frac{\nu+1}{2} \right) }{\sqrt{\pi}(4xz)^{\nu-1}\Gamma\left(\frac{\nu}{2} \right) }\int\limits_{|x-z|}^{x+z}f(y) y [(z^2-(x-y)^2)((x+y)^2-z^2)]^{\frac{\nu}{2} -1} dy
	\end{equation}
		 coincides with the generalized translation operator (see \cite{Levitan1,Levitan2,Levitan3}), for which the next properties are valid
	\begin{equation}
	\label{Pro0}
		\,^\nu T_x^z (B_\nu)_x= (B_\nu)_z\,^\nu T_x^z,
		\end{equation}
		\begin{equation}
		\label{Pro}
		 \,^\nu T_x^zf(x)|_{z=0}=f(x),\qquad  \frac{\partial}{\partial z}\,^\nu T_x^zf(x)\biggr|_{z=0}=0.
		\end{equation}
\end{theorem}
\begin{proof} We have
	$$
	\left( T^{(z)}_{\nu,\, \nu} f\right) (x) =	H_{\nu}^{-1} \left[ j_{\frac{\nu-1}{2}}(z t) H_{\nu}[f](t)\right](x)=
	$$
	$$
	=\frac{2^{1-\nu}}{\Gamma^2\left(\frac{\nu+1}{2}\right)}\int\limits_{0}^{\infty} {j}_{\frac{\nu-1}{2}} (xt)\,
	 j_{\frac{\nu-1}{2}}(z t)\,t^\nu\,dt \int\limits_{0}^{\infty }{j}_{\frac{\nu-1}{2}}(ty) f(y) y^\nu dy=
	$$	
	$$
	=\frac{2^{1-\nu}}{\Gamma^2\left(\frac{\nu+1}{2}\right)}\int\limits_{0}^{\infty}f(y) y^\nu dy \int\limits_{0}^{\infty } {j}_{\frac{\nu-1}{2}} (xt)\,{j}_{\frac{\nu-1}{2}}(ty)\,
	j_{\frac{\nu-1}{2}}(z t)\,t^\nu\,dt.
	$$	
The formula
$$
\int\limits_0^{\infty}j_{\frac{\nu-1}{2}}(t x)j_{\frac{\nu-1}{2}}(t y)
j_{\frac{\nu-1}{2}}(t z)t^\nu dt=$$
$$=
\left\{%
\begin{array}{ll}
$$0$$, & \hbox{$0<y<|x-z|$ or $y>x+z$;} \\
$$\frac{2\Gamma^3\left(\frac{\nu+1}{2} \right) }{\sqrt{\pi}\Gamma\left(\frac{\nu}{2} \right) }	\frac{[(z^2-(x-y)^2)((x+y)^2-z^2)]^{\frac{\nu}{2} -1}}{(xyz)^{\nu-1}}
$$, & \hbox{$|x-z|<y<x+z,$} \\
\end{array}%
\right.
	$$
is true (see formula 2.12.42.14 in \cite{IR2},  p. 204) for $\nu>0$. Therefore we obtain
	$$
	\left( T^{(z)}_{\nu,\, \nu} f\right) (x) =	
	$$
	$$
	=\frac{2^{1-\nu}}{\Gamma^2\left(\frac{\nu+1}{2}\right)}\frac{2\Gamma^3\left(\frac{\nu+1}{2} \right) }{\sqrt{\pi}(xz)^{\nu-1}\Gamma\left(\frac{\nu}{2} \right) }\int\limits_{|x-z|}^{x+z}f(y) y [(z^2-(x-y)^2)((x+y)^2-z^2)]^{\frac{\nu}{2} -1} dy=
	$$
	$$
	=\frac{2^\nu\Gamma\left(\frac{\nu+1}{2} \right) }{\sqrt{\pi}(4xz)^{\nu-1}\Gamma\left(\frac{\nu}{2} \right) }\int\limits_{|x-z|}^{x+z}f(y) y [(z^2-(x-y)^2)((x+y)^2-z^2)]^{\frac{\nu}{2} -1} dy
	=\,^\nu T_x^zf(x).
	$$
From the derived representation it is clear that $\,^\nu T_x^zf(x)=\,^\nu T_z^xf(z)$
and from (\ref{449}) it follows that
$
\,^\nu T_x^z (B_{\nu})_x = (B_{\nu})_x \,^\nu T_x^z,
$ consequently $	\,^\nu T_x^z (B_\nu)_x= (B_\nu)_z\,^\nu T_x^z$.

Properties (\ref{Pro}) follow easily from the representation (\ref{Shift1}). It completes the proof.
\end{proof}

More frequently used representation of generalized translation operator $\,^\nu T_z^x$ is (see \cite{Levitan1,Levitan2,Levitan3})
	\begin{equation}
	\label{Shift1}
\,^\nu T^z_xf(x)=C(\nu)\int\limits_0^\pi
f(\sqrt{x^2+z^2-2xz\cos{\varphi}})\sin^{\nu-1}{\varphi}d\varphi,
	\end{equation}
$$
C(\nu)=\left(\int\limits_0^\pi\sin^{\nu-1}{\varphi}d\varphi\right)^{-1}=
\frac{\Gamma\left(\frac{\nu+1}{2}\right)}{\sqrt{\pi}\,\,\Gamma\left(\frac{\nu}{2}\right)}.
$$
It is easy to see that it is the same as ours.

So it is possible and easy to obtain generalized translation operators by ITCM, and its basic properties follows immediately from ITCM integral representation.

\section{Integral representations of transmutations for perturbed differential  Bessel operators}

Now let us prove a general result on transmutations for perturbed  Bessel operator with potentials. These results are of technical form, they were proved many times for some special cases, it is convenient to prove the general result accurately here. It is the necessary step 2 from ITCM algorithm, it turns operators obtained by ITCM with formal transmutation property into transmutations with exact conditions on input parameters and classes of functions.

Further we will construct a transmutation  operator $S_{\nu,\mu}$ intertwining Bessel operators $B_\nu$ and $B_\mu$. In this case it is reasonable to use
the Hankel transforms of orders $\nu$ and $\mu$ respectively. So for a pair of perturbed Bessel differential operators
$$
A=B_\nu+q(x),\qquad B=B_\mu+r(x)
$$
we seek for a transmutation operator $S_{\nu,\mu}$ such that
\begin{equation}\label{2Bes}
S_{\nu,\mu}(B_\nu+q(x))u=(B_\mu+r(x))S_{\nu,\mu}u.
\end{equation}
Let apply ITCM and obtain it in the form
$$
S_{\nu,\mu}=H^{-1}_\mu \frac{1}{w(t)} H_\nu
$$
with arbitrary $w(t), w(t)\neq 0$.
So we have formally
$$
S_{\nu,\mu}f(x)	=\frac{2^{1-\mu}}{\Gamma^2\left(\frac{\mu+1}{2}\right)}
\int\limits_{0}^{\infty} {j}_{\frac{\mu-1}{2}} (xt)\,
\frac{t^{\mu}}{w(t)}\,dt \int\limits_{0}^{\infty }{j}_{\frac{\nu-1}{2}}(ty) f(y) y^\nu dy=
$$
$$
=\frac{2^{1-\mu}}{\Gamma^2\left(\frac{\mu+1}{2}\right)}
\int\limits_{0}^{\infty}f(y) y^\nu dy \int\limits_{0}^{\infty }{j}_{\frac{\nu-1}{2}}(ty)  {j}_{\frac{\mu-1}{2}} (xt)\,
\frac{t^{\mu}}{w(t)}\,dt.
$$

For all known cases we may represent transmutations $S_{\nu,\mu}$ in the next general form (see  \cite{Sta}, \cite{volk})
$$
S_{\nu,\mu}f(x)=a(x)f(x)+\int\limits_0^x K(x,y)f(y)y^\nu dy+\int\limits_x^\infty L(x,y)f(y)y^\nu dy.
$$
Necessary conditions on  kernels $K$ and $L$ as well as on  functions $a(x), f(x)$ to satisfy \eqref{2Bes} are given in the following theorem.

\begin{theorem}\label{teo1}	Let  $u\in L_2(0,\infty)$ be twice continuously differentiable on $[0,\infty)$ such that $u'(0)=0$,  $q$ and $r$ be such functions that
	$$
	\int\limits_0^\infty t^\delta|q(t)|dt<\infty,\qquad \int\limits_0^\infty t^\varepsilon|r(t)|dt<\infty
	$$
	for some $\delta<\frac{1}{2}$ and $\varepsilon<\frac{1}{2}$.
	When there  exists a transmutation operator of the form
	\begin{equation}\label{Tra}
	S_{\nu,\mu}u(x)=a(x)u(x)+\int\limits_0^x K(x,t)u(t)t^\nu dt+\int\limits_x^\infty L(x,t)u(t)t^\nu dt,
	\end{equation}
	such that
	\begin{equation}\label{Ur}
	S_{\nu,\mu}\biggl[B_\nu+q(x)\biggr]u(x)=\biggl[B_\mu+r(x)\biggr]S_{\nu,\mu}u(x)
		\end{equation}
	with twice continuously differentiable kernels $K(x,t)$ and $L(x,t)$ on $[0,\infty)$ such that
	$$
	\lim\limits_{t\rightarrow 0}t^\nu K(x,t) u'(t)=0,
	\qquad
	\lim\limits_{t\rightarrow 0}t^\nu K_t(x,t)u(t)=0
	$$
	and
	$$
	\lim\limits_{t\rightarrow \infty}t^\nu L(x,t) u'(t)=0,
	\qquad
	\lim\limits_{t\rightarrow \infty}t^\nu L_t(x,t)u(t)=0
	$$
	satisfying the next relations	
	$$
	\biggl[ (B_\nu)_t+q(t)\biggr]  K(x,t)=\biggl[ (B_\mu)_x+r(x)\biggr] K(x,t),
	$$
	$$
	\biggl[ (B_\nu)_t+q(t)\biggr] L(x,t)=\biggl[ (B_\mu)_x+r(x)\biggr]  L(x,t),
	$$
	and
	$$
	a(x)\biggl[B_\nu +q(x)\biggr]u(x)-\biggl[B_\mu+r(x)\biggr] a(x)u(x)
	=
	$$
	$$
	=(\mu+\nu)x^{\nu-1}u(x)\biggl[K(x,x)-L(x,x)\biggr]+2x^\nu u(x)\biggl[ K'(x,x)-L'(x,x)\biggr].
	$$	
\end{theorem}
\begin{proof} First we have
	$$
	S_{\nu,\mu}(B_\nu u(x)+q(x)u(x))=a(x)[B_\nu u(x)+q(x)u(x)]+
	$$
	$$+\int\limits_0^x K(x,t)(B_\nu u(t)+q(t)u(t))t^\nu dt+\int\limits_x^\infty L(x,t)(B_\nu u(t)+q(t)u(t))t^\nu dt.
	$$	
	Substituting Bessel operator in the form $B_\nu=\frac{1}{t^\nu}	\frac{d}{dt}t^\nu\frac{d}{dt}$ and integrating by parts we obtain	
	$$
	\int\limits_0^x K(x,t)(B_\nu u(t))\, t^\nu dt=\int\limits_0^x K(x,t)
	\frac{d}{dt}t^\nu\frac{d}{dt}u(t)dt=
	$$
	$$
	=K(x,t)t^\nu u'(t)\biggr|_{t=0}^x-	
	\int\limits_0^x t^\nu K_t(x,t) \frac{d}{dt}u(t)dt=
	$$
	$$
	=K(x,t)t^\nu u'(t)\biggr|_{t=0}^x-t^\nu K_t(x,t)u(t)\biggr|_{t=0}^x+	
	\int\limits_0^x \left( (B_\nu)_t K_t(x,t)\right)  u(t)t^\nu dt.
	$$
	Since
	$$
	\lim\limits_{t\rightarrow 0}t^\nu K(x,t) u'(t)=0,
	\qquad
	\lim\limits_{t\rightarrow 0}t^\nu K_t(x,t)u(t)=0
	$$
	we obtain
	$$
	\int\limits_0^x K(x,t)(B_\nu u(t))\, t^\nu dt=$$
	$$=K(x,x)x^\nu u'(x)-x^\nu u(x) K_t(x,t)\biggr|_{t=x}+	
	\int\limits_0^x \left( (B_\nu)_t K(x,t)\right)  u(t)t^\nu dt.
	$$
	Similarly
	$$
	\int\limits_x^\infty L(x,t) (B_\nu u(t))\, t^\nu dt=\int\limits_x^\infty L(x,t)\frac{d}{dt}t^\nu\frac{d}{dt}u(t)dt=
	$$
	$$
	=L(x,t)t^\nu u'(t)\biggr|_{t=x}^\infty-	
	\int\limits_x^\infty t^\nu L_t(x,t) \frac{d}{dt}u(t)dt=
	$$
	$$
	=L(x,t)t^\nu u'(t)\biggr|_{t=x}^\infty-t^\nu L_t(x,t)u(t)\biggr|_{t=x}^\infty+	
	\int\limits_x^\infty \left( (B_\nu)_t L(x,t)\right)  u(t)t^\nu dt.
	$$
	Since
	$$
	\lim\limits_{t\rightarrow \infty}t^\nu L(x,t) u'(t)=0,
	\qquad
	\lim\limits_{t\rightarrow \infty}t^\nu L_t(x,t)u(t)=0
	$$
	we obtain
	$$
	\int\limits_x^\infty L(x,t) (B_\nu u(t))\, t^\nu dt=
	$$
	$$
	=-L(x,x)x^\nu u'(x)+x^\nu u(x)L_t(x,t)\biggr|_{t=x}+	
	\int\limits_x^\infty \left( (B_\nu)_t L(x,t)\right)  u(t)t^\nu dt.
	$$
	Therefore
	$$
	S_{\nu,\mu}(B_\nu u(x)+q(x)u(x))=a(x)\biggl[B_\nu +q(x)\biggr]u(x)+
	$$
	$$
	+x^\nu K(x,x) u'(x)-x^\nu u(x) K_t(x,t)\biggr|_{t=x}-x^\nu L(x,x) u'(x)+x^\nu u(x)L_t(x,t)\biggr|_{t=x}+
	$$
	$$+\int\limits_0^x \left( (B_\nu)_t K(x,t)+q(t)K(x,t)\right)  u(t)t^\nu dt+\int\limits_x^\infty \left( (B_\nu)_t L(x,t)+q(t)L(x,t)\right)  u(t)t^\nu dt.
	$$

	Further we have
	$$
	(B_\mu+r(x))S_{\nu,\mu}u(x)=$$
	$$=(B_\mu+r(x))\left(a(x)u(x)+\int\limits_0^x K(x,t)u(t)t^\nu dt+\int\limits_x^\infty L(x,t)u(t)t^\nu dt \right)=
	$$
	$$
	=B_\mu\left[ a(x)u(x)\right]+a(x)r(x)u(x)+B_\mu\int\limits_0^x K(x,t)u(t)t^\nu dt+B_{\mu}\int\limits_x^\infty L(x,t)u(t)t^\nu dt+
	$$
	$$
	+r(x)\int\limits_0^x K(x,t)u(t)t^\nu dt+r(x)\int\limits_x^\infty L(x,t)u(t)t^\nu dt.
	$$
	Using formula of differentiation of  integrals depending on the parameter we get
	$$
	(B_\mu)_x\int\limits_0^x K(x,t)u(t)t^\nu dt=\frac{1}{x^\mu}\frac{d}{dx}x^\mu\frac{d}{dx}\int\limits_0^x K(x,t)u(t)t^\nu dt=
	$$
	$$
	=\frac{1}{x^\mu}\frac{d}{dx}\left( x^{\mu+\nu} K(x,x)u(x)+ x^\mu\int\limits_0^x K_x(x,t)u(t)t^\nu dt\right) =
	$$	
	$$
	=\frac{1}{x^\mu}\biggl((\mu+\nu)x^{\mu+\nu-1} K(x,x)u(x)+x^{\mu+\nu} K'(x,x)u(x)+x^{\mu+\nu} K(x,x)u'(x)+
	$$
	$$
	+\mu x^{\mu-1}\int\limits_0^x K_x(x,t)u(t)t^\nu dt+ x^\mu\frac{d}{dx}\int\limits_0^x K_x(x,t)u(t)t^\nu dt\biggr) =
	$$	
	$$
	=\frac{1}{x^\mu}\biggl((\mu+\nu)x^{\mu+\nu-1} K(x,x)u(x)+x^{\mu+\nu} K'(x,x)u(x)+x^{\mu+\nu} K(x,x)u'(x)+
	$$
	$$
	+x^{\mu+\nu} u(x)K_x(x,t)\biggr|_{t=x}+\mu x^{\mu-1}\int\limits_0^x K_x(x,t)u(t)t^\nu dt+
	x^\mu\int\limits_0^x K_{xx}(x,t)u(t)t^\nu dt\biggr) =
	$$	
	$$
	=(\mu+\nu)x^{\nu-1} K(x,x)u(x)+x^{\nu} K'(x,x)u(x)+x^{\nu} K(x,x)u'(x)+
	$$
	$$
	+ x^{\nu} u(x)K_x(x,t)\biggr|_{t=x}+
	\int\limits_0^x (B_{\mu})_xK(x,t)u(t)t^\nu dt
	$$	
	and also
	$$
	(B_{\mu})_x\int\limits_x^\infty L(x,t)u(t)t^\nu dt=\frac{1}{x^\mu}\frac{d}{dx}x^\mu\frac{d}{dx}\int\limits_x^\infty L(x,t)u(t)t^\nu dt=
	$$	
	$$
	=\frac{1}{x^\mu}\frac{d}{dx}\left(-x^{\mu+\nu} L(x,x)u(x)+ x^\mu\int\limits_x^\infty L_x(x,t)u(t)t^\nu dt\right) =
	$$	
	$$
	=\frac{1}{x^\mu}\biggl(-(\mu+\nu)x^{\mu+\nu-1} L(x,x)u(x)-x^{\mu+\nu} L'(x,x)u(x)-x^{\mu+\nu} L(x,x)u'(x)+
	$$
	$$
	+ \mu x^{\mu-1}\int\limits_x^\infty L_x(x,t)u(t)t^\nu dt+
	x^\mu\frac{d}{dx}\int\limits_x^\infty L_x(x,t)u(t)t^\nu dt\biggr) =
	$$		
	$$
	=\frac{1}{x^\mu}\biggl(-(\mu+\nu)x^{\mu+\nu-1} L(x,x)u(x)-x^{\mu+\nu} L'(x,x)u(x)-x^{\mu+\nu} L(x,x)u'(x)+
	$$
	$$
	-x^{\mu+\nu}u(x)L_x(x,t)\biggr|_{t=x}+ \mu x^{\mu-1}\int\limits_x^\infty L_x(x,t)u(t)t^\nu dt
	+
	x^\mu\int\limits_x^\infty L_{xx}(x,t)u(t)t^\nu dt\biggr)=
	$$		
	$$
	=-(\mu+\nu)x^{\nu-1} L(x,x)u(x)-x^{\nu} L'(x,x)u(x)-x^{\nu} L(x,x)u'(x)-
	$$
	$$
	-x^{\nu}u(x)L_x(x,t)\biggr|_{t=x}+
	\int\limits_x^\infty (B_\mu)_x L(x,t)u(t)t^\nu dt.
	$$	
	So	
	$$
	(B_\mu+r(x))S_{\nu,\mu}u(x)=$$
	$$
	=\biggl[B_\mu+r(x)\biggr] a(x)u(x)+(\mu+\nu)x^{\nu-1}K(x,x)u(x)+x^\nu K'(x,x) u(x)+$$
	$$+x^\nu K(x,x)u'(x)+x^\nu u(x)K_x(x,t)\biggr|_{t=x}-
	$$
	$$
	-(\mu+\nu)x^{\nu-1} L(x,x)u(x)-x^{\nu} L'(x,x)u(x)-x^{\nu} L(x,x)u'(x)-x^{\nu}u(x)L_x(x,t)\biggr|_{t=x}+
	$$
	$$+
	\int\limits_0^x \biggl[(B_\mu)_xK(x,t)+r(x)K(x,t)\biggr]  u(t)t^\nu dt+\int\limits_x^\infty \biggl[ (B_\mu)_xL(x,t)+r(x)L(x,t)\biggr]  u(t)t^\nu dt.
	$$
	Since we shoud have an equality	
	$$
	S_{\nu,\mu}(B_\nu+q(x))=(B_\mu+r(x))S_{\nu,\mu},
	$$
	then equating the corresponding terms in both parts we obtain
	$$
	\int\limits_0^x \biggl[ (B_\nu)_t K(x,t)+q(t)K(x,t)\biggr]   u(t)t^\nu dt=\int\limits_0^x \biggl[(B_\mu)_xK(x,t)+r(x)K(x,t)\biggr]  u(t)t^\nu dt
	$$
	and
	$$
	\int\limits_x^\infty \biggl[ (B_\nu)_t L(x,t)+q(t)L(x,t)\biggr]  u(t)t^\nu dt=\int\limits_x^\infty \biggl[ (B_\mu)_xL(x,t)+r(x)L(x,t)\biggr]  u(t)t^\nu dt.
	$$
	From that we derive two equations
	$$
	\biggl[ (B_\nu)_t+q(t)\biggr] K(x,t)= \biggl[ (B_\mu)_x+r(x)\biggr] K(x,t)
	$$
	and
	$$
	\biggl[ (B_\nu)_t+q(t)\biggr]  L(x,t)=\biggl[(B_\mu)_x+r(x)\biggr]  L(x,t).
	$$
	Because of
	$$
	K'(x,x)=K_x(x,t)\biggr|_{t=x}+ K_t(x,t)\biggr|_{t=x}
	$$
	and
	$$
	L'(x,x)=L_x(x,t)\biggr|_{t=x}+ L_t(x,t)\biggr|_{t=x}
	$$
	we get that
	$$
	a(x)\biggl[B_\nu +q(x)\biggr]u(x)+
	$$
	$$
	+x^\nu K(x,x) u'(x)-x^\nu u(x) K_t(x,t)\biggr|_{t=x}-x^\nu L(x,x) u'(x)+x^\nu u(x)L_t(x,t)\biggr|_{t=x}=
	$$
	$$
	=\biggl[B_\mu+r(x)\biggr] a(x)u(x)+(\mu+\nu)x^{\nu-1}K(x,x)u(x)+x^\nu K'(x,x) u(x)+$$
	$$+x^\nu K(x,x)u'(x)+x^\nu u(x)K_x(x,t)\biggr|_{t=x}-
	$$
	$$
	-(\mu+\nu)x^{\nu-1} L(x,x)u(x)-x^{\nu} L'(x,x)u(x)-x^{\nu} L(x,x)u'(x)-x^{\nu}u(x)L_x(x,t)\biggr|_{t=x}
	$$
	which is equvivalent to
	$$
	a(x)\biggl[B_\nu +q(x)\biggr]u(x)-\biggl[B_\mu+r(x)\biggr] a(x)u(x)
	=
	$$
	$$
	=(\mu+\nu)x^{\nu-1}u(x)\biggl[K(x,x)-L(x,x)\biggr]+2x^\nu u(x)\biggl[ K'(x,x)-L'(x,x)\biggr].
	$$
	It completes the proof of the theorem.
\end{proof}

Now consider some special cases.


We consider here some special cases of a transmutation operator $S_{\nu,\mu}$ for functions $q$ and $r$  from the theorem \ref{teo1}. Let  functions $u,q,r$ satisfy the condition of the theorem \ref{teo1}.


1. For the transmutation  in (\ref{Tra}) of the next form
$$
S_{\nu,\mu}u(x)=a(x)u(x)+\int\limits_0^x K(x,t)u(t)t^\nu dt,
$$
with intertwining property
$$
S_{\nu,\mu}\biggl[B_\nu+q(x)\biggr]u(x)=\biggl[B_\mu+r(x)\biggr]S_{\nu,\mu}u(x)
$$
a kernel $K(x,t)$ and function $a(x)$  should satisfy  relations	
$$
\biggl[ (B_\nu)_t+q(t)\biggr]  K(x,t)=\biggl[ (B_\mu)_x+r(x)\biggr] K(x,t),
$$
and
$$
a(x)\biggl[B_\nu +q(x)\biggr]u(x)-\biggl[B_\mu+r(x)\biggr] a(x)u(x)
=
$$
$$
=(\mu+\nu)x^{\nu-1}u(x)K(x,x)+2x^\nu u(x)K'(x,x).
$$	
In the particular case when $\nu=\mu,$ $r(x)=0$, $a(x)=1$  transmutations with such representations were obtained in \cite{Sta,volk}.

2. For the transmutation  in (\ref{Tra}) of the next form
$$
S_{\nu,\mu}u(x)=a(x)u(x)+\int\limits_x^\infty L(x,t)u(t)t^\nu dt,
$$
such that
$$
S_{\nu,\mu}(B_\nu+q(x))u=(B_\mu+r(x))S_{\nu,\mu}u
$$
a kernel $L(x,t)$ and function $a(x)$  should satisfy to relations		
$$
\biggl[ (B_\nu)_t+q(t)\biggr] L(x,t)=\biggl[ (B_\mu)_x+r(x)\biggr]  L(x,t)
$$
and
$$
a(x)\biggl[B_\nu +q(x)\biggr]u(x)-\biggl[B_\mu+r(x)\biggr] a(x)u(x)
=
$$
$$
=-(\mu+\nu)x^{\nu-1}u(x)L(x,x)-2x^\nu u(x)L'(x,x).
$$	


3. When one  potential in (\ref{Ur}) is equal to zero we get  for a transmutation operator
$$
S_{\nu,\mu}u(x)=a(x)u(x)+\int\limits_0^x K(x,t)u(t)t^\nu dt+\int\limits_x^\infty L(x,t)u(t)t^\nu dt,
$$
such that
$$
S_{\nu,\mu}\biggl[B_\nu+q(x)\biggr]u(x)=B_\mu S_{\nu,\mu}u(x)
$$
the next conditions on kernels $K(x,t)$, $L(x,t)$ and function $a(x)$ 		
$$
\biggl[ (B_\nu)_t+q(t)\biggr]  K(x,t)=(B_\mu)_x K(x,t),
$$
$$
\biggl[ (B_\nu)_t+q(t)\biggr] L(x,t)= (B_\mu)_x  L(x,t),
$$
and
$$
a(x)\biggl[B_\nu +q(x)\biggr]u(x)-B_\mu[a(x)u(x)]
=
$$
$$
=(\mu+\nu)x^{\nu-1}u(x)\biggl[K(x,x)-L(x,x)\biggr]+2x^\nu u(x)\biggl[ K'(x,x)-L'(x,x)\biggr].
$$	

4.  When $\mu=0$ and $r(x)\equiv 0$  in (\ref{Ur}) for a transmutation operator
$$
S_{\nu}u(x)=a(x)u(x)+\int\limits_0^x K(x,t)u(t)t^\nu dt+\int\limits_x^\infty L(x,t)u(t)t^\nu dt,
$$
such that
$$
S_{\nu}(B_\nu+q(x))u=D^2 S_{\nu}u
$$
kernels $K(x,t)$, $L(x,t)$ and function $a(x)$  should satisfy to relations	
$$
\biggl[ (B_\nu)_t+q(t)\biggr]  K(x,t)=D^2_x K(x,t),
$$
$$
\biggl[ (B_\nu)_t+q(t)\biggr] L(x,t)=D^2_x L(x,t),
$$
and
$$
a(x)\biggl[B_\nu +q(x)\biggr]u(x)-D^2_x a(x)u(x)
=
$$
$$
=\nu x^{\nu-1}u(x)\biggl[K(x,x)-L(x,x)\biggr]+2x^\nu u(x)\biggl[ K'(x,x)-L'(x,x)\biggr].
$$	


5. When both potentials  in (\ref{Ur}) are equal to zero for a transmutation operator
$$
S_{\nu,\mu}u(x)=a(x)u(x)+\int\limits_0^x K(x,t)u(t)t^\nu dt+\int\limits_x^\infty L(x,t)u(t)t^\nu dt,
$$
such that
$$
S_{\nu,\mu}B_\nu u=B_\mu S_{\nu,\mu}u
$$
kernels $K(x,t)$, $L(x,t)$ and function $a(x)$  should satisfy to relations		
$$
(B_\nu)_t  K(x,t)= (B_\mu)_x K(x,t),
$$
$$
(B_\nu)_t L(x,t)= (B_\mu)_x  L(x,t),
$$
and
$$
a(x)B_\nu u(x)- B_\mu[ a(x)u(x)]
=
$$
$$
=(\mu+\nu)x^{\nu-1}u(x)\biggl[K(x,x)-L(x,x)\biggr]+2x^\nu u(x)\biggl[ K'(x,x)-L'(x,x)\biggr].
$$


6. When both potentials  in (\ref{Ur}) are equal to zero and $\mu=\nu$ for a transmutation operator
$$
S_{\nu,\nu}u(x)=a(x)u(x)+\int\limits_0^x K(x,t)u(t)t^\nu dt+\int\limits_x^\infty L(x,t)u(t)t^\nu dt,
$$
such that
$$
S_{\nu,\nu}B_\nu u=B_\nu S_{\nu,\nu}u
$$
kernels $K(x,t)$, $L(x,t)$ and function $a(x)$  should satisfy to relations			
$$
(B_\nu)_t K(x,t)= (B_\nu)_x K(x,t),
$$
$$
(B_\nu)_t L(x,t)= (B_\nu)_x  L(x,t),
$$
and
$$
a(x)B_\nu u(x)-B_\nu[a(x)u(x)]
=
$$
$$
=2\nu x^{\nu-1}u(x)\biggl[K(x,x)-L(x,x)\biggr]+2x^\nu u(x)\biggl[ K'(x,x)-L'(x,x)\biggr].
$$


7.  When both potentials  in (\ref{Ur}) are equal to zero and $\mu=0$ for a transmutation operator
$$
S_{\nu,0}u(x)=a(x)u(x)+\int\limits_0^x K(x,t)u(t)t^\nu dt+\int\limits_x^\infty L(x,t)u(t)t^\nu dt,
$$
such that
$$
S_{\nu,0}B_\nu u=D^2 S_{\nu,0}u
$$
kernels $K(x,t)$, $L(x,t)$ and function $a(x)$  should satisfy to relations	
$$
(B_\nu)_t K(x,t)= D^2_x  K(x,t),
$$
$$
(B_\nu)_t  L(x,t)= D^2_x L(x,t)
$$
and
$$
a(x)B_\nu u(x)-D^2_x[ a(x)u(x)]
=
$$
$$
=\nu x^{\nu-1}u(x)\biggl[K(x,x)-L(x,x)\biggr]+2x^\nu u(x)\biggl[ K'(x,x)-L'(x,x)\biggr].
$$

\section{Application of  transmutations obtained by ITCM to integral representations of solutions to
 hyperbolic equations with Bessel operators}\label{APPL}

Let us solve the problem of obtaining transmutations by ITCM (step 1) and justify integral representation and proper function classes for it (step 2). Now consider applications of these transmutations to integral representations of solutions to  hyperbolic equations with Bessel operators (step 3). For simplicity we consider model equations, for them integral representations of solutions are mostly known.  More complex problems need more detailed and spacious calculations. But even for these model problems considered below application of the transmutation method based on ITCM is new, it allows more unified and simplified approach to hyperbolic equations with Bessel operators of EPD/GEPD types.

\subsection{Application of transmutations for finding general solution to EPD type equation}

Standard approach to solving differential equations is to find its general solution first, and then substitute given functions to find particular solutions. Here we will show how to obtain general solution of EPD type equation using transmutation operators.

\vskip 0.5cm

{\it Proposition 1.} General solution of the equation
\begin{equation}\label{VolnBes}
\frac{\partial^2 u}{\partial x^2}=(B_\mu)_t u,\qquad u=u(x,t;\mu)
\end{equation}
for $ 0 <\mu <1 $ is represented in the form
\begin{equation}\label{VolnBesResh}
u=\int\limits_{0}^{1}\frac{\Phi(x+t(2p-1))}
{(p(1-p))^{1-\frac{\mu}{2}}}
\,dp+t^{1-\mu}\int\limits_{0}^{1}\frac{\Psi(x+t(2p-1))}
{(p(1-p))^{\frac{\mu}{2}}}
\,dp,
\end{equation}
with a pair of arbitrary functions $\Phi, \Psi$.

\begin{proof}
First, we consider wave equation (\ref{Wave1}) when $a=1$
\begin{equation}\label{Voln}
\frac{\partial^2 u}{\partial t^2}=\frac{\partial^2 u}{\partial x^2}.
\end{equation}
General solution to this equation  has the form
\begin{equation}\label{VolnResh}
F(x+t)+G(x-t),
\end{equation}
where $F$ and $G$ are arbitrary functions. Applying operator  (\ref{Poisson}) (obtained by ITCM in section 4 !) by variable $t$ we obtain that one  solution to the equation (\ref{VolnBes})
is
$$
u_1=2C(\mu)\frac{1}{t^{\mu-1}}
\int\limits_0^{t}[F(x+z)+G(x-z)]
(t^2-z^2)^{\frac{\mu}{2}-1}\,dz.
$$
Let transform the resulting general solution as follows
$$
u_1=\frac{C(\mu)}{t^{\mu-1}}
\int\limits_{-t}^{t}\frac{F(x+z)+F(x-z)+G(x+z)+G(x-z)}{(t^2-z^2)^{1-\frac{\mu}{2}}}
\,dz.
$$
Introducing a new variable $p$ by formula $z=t(2p-1)$ we get
$$
u_1=\int\limits_{0}^{1}\frac{\Phi(x+t(2p-1))}
{(p(1-p))^{1-\frac{\mu}{2}}}
\,dp,
$$
where
$$
\Phi(x+z){=}\left[F(x+z){+}F(x-z){+}G(x+z){+}G(x-z)\right]
$$
is an arbitrary function.

It is easy to see that if $u(x,t;\mu)$ is a solution of (\ref{VolnBes}) then a function $t^{1-\mu}u(x,t;2{-}\mu)$ is also a solution of  (\ref{VolnBes}). Therefore the second solution to (\ref{VolnBes}) is
$$
u_2=t^{1-\mu}\int\limits_{0}^{1}\frac{\Psi(x+t(2p-1))}
{(p(1-p))^{\frac{\mu}{2}}}
\,dp,
$$
where $\Psi$ is an arbitrary function, not coinciding with $\Phi$.
Summing $u_1$ and $u_2$ we obtain general solution to (\ref{VolnBes}) of the form (\ref{VolnBesResh}).
From the (\ref{VolnBesResh}) we can see that for summable functions $\Phi$ and $\Psi$ such a solution exists for $ 0 <\mu <1 $.
\end{proof}

\subsection{Application of transmutations for finding general solution to GEPD type equation}

Now we derive a general solution to GEPD type equation by transmutation method.

{\it Proposition 2.} General solution to the equation
\begin{equation}\label{VolnBes1}
(B_\nu)_x u=(B_\mu)_t u,\qquad u=u(x,t;\nu,\mu)
\end{equation}
for $ 0 <\mu <1 $, $ 0 <\nu <1 $  is
	$$
	u=\frac{2\Gamma\left(\frac{\nu+1}{2}\right)}{\sqrt{\pi}
		\Gamma\left(\frac{\nu}{2}\right)}\left(x^{1-\nu}\,\int\limits_{0}^{x}(x^2- y^2)^{\frac{\nu}{2}-1}dy \int\limits_{0}^{1}\frac{\Phi(y+t(2p-1))}
	{(p(1-p))^{1-\frac{\mu}{2}}}
	\,dp+\right.
	$$
\begin{equation}\label{VolnBesResh1}
\left. +t^{1-\mu}x^{1-\nu}\,\int\limits_{0}^{x}(x^2- y^2)^{\frac{\nu}{2}-1}dy \int\limits_{0}^{1}\frac{\Psi(y+t(2p-1))}
	{(p(1-p))^{\frac{\mu}{2}}}
	\,dp.\right)
\end{equation}
\begin{proof}
	Applying Poisson operator (\ref{Poisson}) (again obtained by ITCM in section 4 !) with index $\nu$ by variable $x$ to the (\ref{VolnBesResh}) we derive general solution (\ref{VolnBesResh1})
to the equation (\ref{VolnBes1}).
\end{proof}

\subsection{Application of transmutations for finding general solution to GEPD type equation with spectral parameter}

Now let apply transmutations for finding general solution to GEPD type equation with spectral parameter.

{\it Proposition 3.} General solution to the equation
\begin{equation}\label{VolnBes2}
(B_\nu)_x u=(B_\mu)_t u+b^2u,\qquad u=u(x,t;\nu,\mu)
\end{equation}
for $ 0 <\mu <1 $, $ 0 <\nu <1 $  is
$$
u{=}\frac{2\Gamma\left(\frac{\nu+1}{2}\right)}{\sqrt{\pi}
	\Gamma\left(\frac{\nu}{2}\right)}\left(x^{1-\nu}\,\int\limits_{0}^{x}(x^2- y^2)^{\frac{\nu}{2}-1}dy \int\limits_{0}^{1}\frac{\Phi(y+t(2p-1))}
{(p(1-p))^{1-\frac{\mu}{2}}}j_{\frac{\mu}{2}-1}(2bt\sqrt{p(1-p)})
\,dp+\right.
$$
\begin{equation}\label{VolnBesResh3}
\left. +t^{1-\mu}x^{1-\nu}\,\int\limits_{0}^{x}(x^2- y^2)^{\frac{\nu}{2}-1}dy \int\limits_{0}^{1}\frac{\Psi(y+t(2p-1))}
{(p(1-p))^{\frac{\mu}{2}}}j_{-\frac{\mu}{2}}(2bt\sqrt{p(1-p)})
\,dp.\right)
\end{equation}
\begin{proof}
General solution to the equation
$$
\frac{\partial^2 u}{\partial x^2}=(B_\mu)_tu+b^2u,\qquad u=u(x,t;\mu),\qquad 0<\mu<1
$$
is (see \cite{Polyanin}, p. 328)
$$
u=\int\limits_{0}^{1}\frac{\Phi(x+t(2p-1))}
{(p(1-p))^{1-\frac{\mu}{2}}}j_{\frac{\mu}{2}-1}(2bt\sqrt{p(1-p)})
\,dp+$$
$$+t^{1-\mu}\int\limits_{0}^{1}\frac{\Psi(x+t(2p-1))}
{(p(1-p))^{\frac{\mu}{2}}}j_{-\frac{\mu}{2}}(2bt\sqrt{p(1-p)})
\,dp.
$$
Applying Poisson operator (\ref{Poisson}) (again obtained by ITCM in section 4 !) with index $\nu$ by variable $x$ to the (\ref{VolnBesResh}) we derive general solution (\ref{VolnBesResh1})
to the equation (\ref{VolnBes1}).
\end{proof}

\subsection{Application  of transmutations for finding general solutions to singular
Cauchy problems}

Using (\ref{VolnBesResh}) for $0<\mu<1$ we  find the solution of the Cauchy problem
\begin{equation}\label{VolnBesCoshy}
\frac{\partial^2 u}{\partial x^2}=(B_\mu)_t u,\qquad u=u(x,t;\mu),\qquad 0<\mu<1,
\end{equation}
\begin{equation}\label{VolnBesCoshyUs}
u(x,0;\mu)=f(x),\qquad  \left( t^\mu \frac{\partial u}{\partial t}\right)\biggr|_{t=0} =g(x).
\end{equation}
and this solution is
\begin{equation}\label{ReshVolnBesCoshy}
u=\frac{\Gamma(\mu)}{\Gamma^2\left(\frac{\mu}{2}\right) }\int\limits_0^1 \frac{f(x+t(2p-1))}{(p(1-p))^{1-\frac{\mu}{2}}}dp
+t^{1-\mu}\,\frac{\Gamma(\mu+2)}{(1-\mu)\Gamma^2\left(\frac{\mu}{2}+1\right) }\,\int\limits_0^1
\frac{g(x+t(2p-1))}{(p(1-p))^{\frac{\mu}{2}}} dp.
\end{equation}

The solution of the Cauchy problem
\begin{equation}\label{VolnBesCoshy1}
\frac{\partial^2 u}{\partial x^2}=(B_\mu)_t u,\qquad u=u(x,t;\mu),
\end{equation}
\begin{equation}\label{VolnBesCoshyUs1}
u(x,0;\mu)=f(x),\qquad  \left(\frac{\partial u}{\partial t}\right)\biggr|_{t=0} =0
\end{equation}
exists for any $\mu>0$ and has the form
\begin{equation}\label{VolnBesCoshyResh1}
u(x,t;\mu)=\frac{\Gamma(\mu)}{\Gamma^2\left(\frac{\mu}{2}\right) }\int\limits_0^1 \frac{f(x+t(2p-1))}{(p(1-p))^{1-\frac{\mu}{2}}}dp.
\end{equation}
Taking into account (\ref{InC}) we can see that it is possible to obtain solution of (\ref{VolnBesCoshy1}--\ref{VolnBesCoshyResh1}) applying Poisson operator to the solution of the Cauchy problem (\ref{Voln}--\ref{VolnBesCoshyResh1}) directly.

The solution of the Cauchy problem
\begin{equation}\label{VolnBesCoshy2}
\frac{\partial^2 u}{\partial x^2}=(B_\mu)_t u,\qquad u=u(x,t;\mu),
\end{equation}
\begin{equation}\label{VolnBesCoshyUs2}
u(x,0;\mu)=0,\qquad   \left( t^\mu \frac{\partial u}{\partial t}\right)\biggr|_{t=0} =g(x).
\end{equation}
exists for any $\mu<1$  and has the form
$$
u(x,t;\mu)=t^{1-\mu}\,\frac{\Gamma(\mu+2)}{(1-\mu)\Gamma^2\left(\frac{\mu}{2}+1\right) }\,\int\limits_0^1
\frac{g(x+t(2p-1))}{(p(1-p))^{\frac{\mu}{2}}} dp.
$$

The Cauchy problem (\ref{VolnBesCoshy}--\ref{VolnBesCoshyUs}) can be considered for $ \mu\notin(0,1)$.
In this case, to obtain the solution the transmutation operator (\ref{Theo1}) (obtained in section 4 by ITCM !) should be used. The case $\mu=-1,-3,-5,...$ is exceptional and has to be studied separately.

It is easy to see that if we know that generalised translation (obtained in section 4 by ITCM !) has properties
(\ref{Pro0}--\ref{Pro}) we can in a straightway obtain that the solution to the equation
$$
(B_\mu)_x u=(B_\mu)_t u,\qquad u=u(x,t;\mu)
$$
with initial conditions
$$u(x,0)=f(x),\qquad u_t(x,0)=0$$
 is $u=\label{key}\,^\mu T_x^tf(x)$.
Now  the first and the second descent operators (\ref{OPDBess}) and (\ref{desent}) (obtained in section 4 by ITCM !)  allow to
represent the solution to the Cauchy problem
\begin{equation}\label{VolnBesCoshy4}
(B_\mu)_x u=(B_\nu)_t u,\qquad u=u(x,t;\mu,\nu),
\end{equation}
\begin{equation}\label{VolnBesCoshyUs4}
u(x,0;\mu,\nu)=f(x),\qquad  \left(\frac{\partial u}{\partial t}\right)\biggr|_{t=0} =0.
\end{equation}
For $0<\mu<\nu$ a solution of (\ref{VolnBesCoshy4}--\ref{VolnBesCoshyUs4}) is derived by using (\ref{OPDBess}) and has the form
\begin{equation}\label{SolEx1}
u(x,t;\mu,\nu){=}\frac{ 2\Gamma\left(\frac{\nu+1}{2}\right)}{\Gamma\left(\frac{\nu-\mu}{2}\right)
	\Gamma\left(\frac{\mu+1}{2}\right)}\,t^{1-\nu}\,
\int\limits_0^t(t^2-y^2)^{\frac{\nu-\mu}{2}-1}\,^\mu T^{y}_xf(x)\,y^\mu
dy.
\end{equation}
In the case $0{<}\nu {<} \mu$ by using
(\ref{desent}) we get a solution to (\ref{VolnBesCoshy4}--\ref{VolnBesCoshyUs4}) in the form
\begin{equation}
\label{SolEx2}
u(x,t;\mu,\nu){=} =\frac{2\Gamma\left(\mu-\nu\right)}
{\Gamma^2\left(\frac{\mu-\nu}{2}\right)}
\,\int\limits_{t}^{\infty }
(y^2-t^2)^{\frac{\mu-\nu}{2}-1} \,^\mu T^{y}_xf(x)\, ydy.
\end{equation}

Let consider a Cauchy problem for the  equation
\begin{equation}\label{VolnBes3}
(B_\mu)_x u=(B_\nu)_t u+b^2u,\qquad u=u(x,t;\nu,\mu)
\end{equation}
\begin{equation}\label{Cond}
u(x,0;\nu,\mu)=f(x),\qquad u_t(x,0;\nu,\mu)=0.
\end{equation}
for $ 0 <\mu <1 $, $ 0 <\nu <1 $.
Applying (\ref{Shift}), descent operators (\ref{OPDBess})  and  (\ref{desent})
 we obtain that solution of (\ref{VolnBes3}--\ref{Cond})
in the case $0<\mu<\nu$ is
\begin{equation}\label{Sol01}
u(x,t;\mu,\nu)=\frac{2\Gamma\left(\frac{\nu+1}{2}\right)}
{\Gamma\left(\frac{\nu-\mu}{2}\right)\Gamma\left(\frac{\mu+1}{2}\right)}\times$$
$$\times\,t^{1-\nu}\int\limits_0^t (t^2-y^2)^{\frac{\nu-\mu}{2}-1} j_{\frac{\nu-\mu}{2}-1}\left(b\sqrt{t^2-y^2}\right) \,^\mu T^{y}_xf(x)y^{\mu}dy.
\end{equation}
and in the case $0<\nu < \mu$ is
$$
u(x,t;\mu,\nu){=}
$$
\begin{equation}
\label{Sol02}
 =\frac{2\Gamma\left(\mu-\nu\right)}
{\Gamma^2\left(\frac{\mu-\nu}{2}\right)}
\,\int\limits_{t}^{\infty }
(y^2-t^2)^{\frac{\mu-\nu}{2}-1}j_{\frac{\mu-\nu}{2}-1}\left(b\sqrt{t^2-y^2}\right) \,^\mu T^{y}_xf(x)\, ydy.
\end{equation}


\begin{thebibliography}{99}
	
	

	
	
		\bibitem{Bajlekova0}  E. G. Bajlekova;  \emph{ Fractional evolution equations in Banach spaces},   Technische Universiteit Eindhoven, Thesis, 2001.
		
		\bibitem{Bajlekova1} E. G. Bajlekova;  \emph{Subordination principle for fractional evolution equations}, Fractional Calculus and Applied Analysis, 3 (3) (2000), 213--230
	
		\bibitem{Lesha1} A. V. Borovskikh;  \emph{The formula for propagating waves for a one-dimensional inhomogeneous medium},  Differ. Eq., Minsk, 38 (6) (2002), 758--767.
		
		\bibitem{Lesha2} A. V. Borovskikh;  \emph{The method of propagating waves},
		Trudy seminara I.G. Petrovskogo: Moscow, 24 (2004), 3--43.
	
		\bibitem{Bresters2}  D. W.  Bresters;  \emph{On a Generalized Euler--Poisson--Darboux Equation}, SIAM J. Math. Anal., 9 (5) (1978),  924--934.
	
	\bibitem{CSh} R. W. Carroll, R.E. Showalter;  \emph{Singular and Degenerate Cauchy problems},  Academic Press, New York, 1976.

\bibitem{Car1} R. W. Carroll;  \emph{Transmutation and Operator Differential Equations},  North Holland,  1979.

\bibitem{Car2} R. W. Carroll;  \emph{Transmutation, Scattering Theory and Special Functions},  North Holland, 1982.

\bibitem{Car3} R. W. Carroll;  \emph{Transmutation Theory and Applications}, North Holland, 1985.


	
		\bibitem{Castillo1}
		R. Castillo-P\'{e}rez, V. V. Kravchenko and S. M. Torba;  \emph{Spectral parameter power series for perturbed Bessel equations}, Appl.
		Math. Comput., 220 (2013), 676--694.

\bibitem{CFH}  H. Chebli, A. Fitouhi, M.M. Hamza;   \emph{Expansion in series of Bessel functions and transmutations for perturbed Bessel operators},  J. Math. Anal. Appl, 181 (3) (1994), 789--802.
	
	\bibitem{Curant}  R. Courant, D. Hilbert; \emph{Methods of Mathematical Physics, vol II}, Interscience (Wiley) New York, 1962.
	
		\bibitem{Darboux} G. Darboux; \emph{  Le\c{c}ons sur la th\'{e}orie g\'{e}n\'{e}rale des surfaces et les applications g\'{e}om\'{e}triques du calcul
		infinit\'{e}simal, vol. 2}, 2nd edn, Gauthier--Villars, Paris, 1915.
		
		\bibitem{Diaz}  J. B. Diaz,  H. F.  Weinberger; \emph{A solution of the singular initial value problem for the Euler--Poisson--equation},  Proc. Amer. Math. Soc., 4 (1953), 703--715.
		
		\bibitem{Dim} I. Dimovski;  \emph{Convolutional Calculus}. Springer, 1990.

\bibitem{EIK} S.D. Eidelman,  S.D. Ivasyshen, A.N. Kochubei, \emph{Analytic Methods In The Theory Of Differential And Pseudo--Differential Equations Of Parabolic Type}, Springer, 2004.
		
		\bibitem{Euler} L. Euler; \emph{ Institutiones calculi integralis}, Opera Omnia, Leipzig, Berlin, 1 (13) (1914), 212--230.
	
	\bibitem{Evans} L. Evans; \emph{Partial Differential Equations}, American Mathematical Society, Providence, 1998.

\bibitem{FH} A. Fitouhi, M.M. Hamza, \emph{Uniform expansion for eigenfunction of singular second order differential operator}, SIAM J. Math. Anal. 21 (6) (1990), 1619--1632.
	
	\bibitem{Fox} D.N. Fox; \emph{The solution and Huygens’ principle for a singular Cauchy problem},
	J. Math. Mech., 8 (1959), 197--219.

\bibitem{Rad2} M. Ghergu, V. Radulescu; \emph{Singular Elliptic Problems. Bifurcation and Asymptotic Analysis}, Oxford Lecture Series in Mathematics and Its Applications, vol. 37, Oxford University Press, 2008.


\bibitem{GKSh} 	A.V. Glushak, V.I. Kononenko, S.D. Shmulevich;  \emph{A Singular Abstract Cauchy Problem}, Soviet Mathematics (Izvestiya VUZ. Matematika), 30 (6) 1986, 78--81.	
	

	\bibitem{Glushak0} 	A.V. Glushak;  \emph{The Bessel Operator Function}, Dokl. Rus. Akad. Nauk,
	352 (5) (1997), 587--589.
	
	\bibitem{Glushak1}  A. V. Glushak, O. A. Pokruchin;  \emph{ A criterion on solvability of the Cauchy problem for an abstract Euler--Poisson--Darboux equation}, Differential Equations,  52 (1) (2016),  39--57.
	
	\bibitem{Glushak2} A. V. Glushak;  \emph{ Abstract Euler--Poisson--Darboux equation with nonlocal condition}, Russian Mathematics (Izvestiya VUZ. Matematika), 60 (6) (2016), 21--28.
	
	\bibitem{Glushak3} A. V. Glushak, V. A. Popova;   \emph{Inverse problem for Euler--Poisson--Darboux abstract differential equation}, Journal of Mathematical Sciences,  149 (4) (2008), 1453--1468.

\bibitem{Gul1} V.S. Guliev, \textit{Sobolev theorems for B--Riesz potentials}, Doklady
of the Russian Academy of Sciences, 358 (4) (1998), 450--451.

	
		\bibitem{Hromov} A. P. Hromov; \emph{ Finite--dimensional perturbations of Volterra operators},  Modern mathematics. Fundamental directions, 10 (2004), 3--163.
	
			\bibitem{Sit0} V. V. Katrakhov and S. M. Sitnik; \emph{Factorization method in transmutation operators theory}
			In Memoria of  Boris Alekseevich Bubnov: nonclassical equations and equations of mixed type. ( editor V. N. Vragov), Novosibirsk, (1990), 104--122.

			\bibitem{Sit01} V. V. Katrakhov and S. M. Sitnik; \emph{Composition method of construction of B-elliptic, B-parabolic and B-hyperbolic transmutation operators}, Doklades of the Russian Academy of Sciences,  337 (3) (1994), 307--311.
		
		\bibitem{kipr} I. A. Kipriyanov; \emph{ Singular Elliptic Boundary Value Problems},
		Nauka, Moscow, 1997.
		
\bibitem{KiprIv}   Kipriyanov I.A., Ivanov L.A.,  \textit{Riezs potentials on the Lorentz spaces}, Mat. Sb., 130(172) 4(8) (1986), 465--474.
		
			
			\bibitem{Kir1} V. S. Kiryakova;  \emph{Generalized fractional calculus and applications}. Pitman Res Notes Math
			301, Longman Scientific \& Technical; Harlow, Co-published with John Wiley, New York, 1994.
			
			\bibitem{KTS}
			V. V. Kravchenko, S. M. Torba and J. Y. Santana-Bejarano;  \emph{Generalized qave polynomials and transmutations
			related to perturbed Bessel equations}, arXiv:1606:07850.
		
			\bibitem{KTC2017}
			V. V. Kravchenko, S. M. Torba and R. Castillo-P\'{e}rez;  \emph{A Neumann series of Bessel functions
			representation for solutions of perturbed Bessel equations},
			Applicable Analysis, 97 5 (2018), 677--704.
		
	\bibitem{Levitan1} B. M. Levitan; \emph{ Generalized translation operators and some of their applications},  Moscow, 1962.
	
	\bibitem{Levitan2} B. M. Levitan; \emph{The theory of generalized shift operators}, Nauka, Moscow, 1973.
	
	\bibitem{Levitan3}	 B. M. Levitan; \emph{ The application of generalized displacement operators to linear differential equations of the second order}, Uspekhi Mat. Nauk.  4:1 (29) (1949), 3--112.

\bibitem{Lya1} L.N. Lyakhov, \textit{Inversion of the B-Riesz potentials}, Dokl. Akad. Nauk SSSR, 321 (3)
(1991), 466--469.

	
		\bibitem{LPSh1} L. N. Lyakhov,  I. P. Polovinkin, E. L. Shishkina; \emph{On a Kipriyanov problem for a singular ultrahyperbolic equation},
		Differ. Equ., 50 (4) (2014), 513--525.	
		
		\bibitem{LPSh2} L. N. Lyakhov,  I. P. Polovinkin, E. L. Shishkina; \emph{Formulas for the Solution of the Cauchy Problem for a Singular Wave Equation with Bessel Time Operator}, Doklady Mathematics, 90 (3) (2014),  737--742.
	
	
	\bibitem{McBArt} A. C. McBride; \emph{Fractional powers of a class of ordinary differential operators}, Proc. London Math. Soc., 3 45 (1982), 519--546.

\bibitem{Rad3} Editors: E. Mitidieri, J. Serrin, V. Radulescu; \emph{Recent Trends in Nonlinear Partial Differential Equations I: Evolution Problems}, Contemporary Mathematics, vol. 594, American Mathematical Society,2013.

\bibitem{Rad4} Editors: E. Mitidieri, J. Serrin, V. Radulescu; \emph{Recent Trends in Nonlinear Partial Differential Equations II: Stationary Problems}, Contemporary Mathematics, vol. 595, American Mathematical Society,2013.

\bibitem{Nogin} Nogin V.A.,   Sukhinin E.V., \textit{Inversion and characterization of hyperbolic potentials in $L_p$-spaces},
Dokl. Acad. Nauk, 329 (5) (1993), 550--552.

	
\bibitem{OZK} H. Ozaktas, Z. Zalevsky, M. Kutay, \emph{The Fractional Fourier Transform: with Applications in Optics and Signal Processing}, Wiley, 2001.
	
		\bibitem{Poisson} S. D. Poisson; \emph{  M\'{e}moire sur l'int\'{e}gration des \'{e}quations lin\'{e}aires aux diff\'{r}ences partielles}, J. de L'\'{E}cole
		Polytechechnique,  1 (19) (1823), 215--248.
		
			\bibitem{Poisson1821} S.  D.  Poisson; \emph{ Sur l'intégration des équations linéaires
			aux différences partielles},   Journal  de  l'École  Royale  Polytechnique,    12 19   (1823),  215--224.
		
		\bibitem{Polyanin} A. D. Polyanin; \emph{Handbook of Linear Partial Differential Equations for Engineers and Scientists}, Chapman \& Hall/CRC Press, Boca Raton–London, 2002.
		
			\bibitem{IR2} A. P.	Prudnikov, Y. A. Brychkov  and O. I. Marichev; \emph{Integrals and Series, Vol. 2, Special Functions}, Gordon \& Breach Sci. Publ., New York, 1990.
			
			\bibitem{Jan}  J. Pruss ;  \emph{ Evolutionary Integral Equations and Applications}, Springer, 2012.

\bibitem{Rad1} V. Radulescu, D. Repovs;  \emph{Partial Differential Equations with Variable Exponents: Variational Methods and Qualitative Analysis}, CRC Press, Taylor \& Francis Group, 2015.

		
			\bibitem{Riman} B. Riemann; \emph{  On the Propagation of Flat Waves of Finite Amplitude}, Ouvres,  Moscow--Leningrad, (1948), 376--395.

\bibitem{Riesz} M. Riesz, \emph{L'integrale de Riemann--Liouville et le probleme de Cauchy},  Acta Mathematica,  81 (1949), 1--223.

\bibitem{SKM} S.G. Samko, A.A. Kilbas, O.L. Marichev, \textit{Fractional integrals and derivatives}, Gordon and Breach Science Publishers, Amsterdam, 1993.

\bibitem{Samp} C.H. Sampson \textit{A characterization of parabolic Lebesgue spaces}, Thesis. Rice Univ., 1968. 	

\bibitem{Shi1}  Shishkina E.L., \textit{On the boundedness of hyperbolic Riesz B--potential},
Lithuanian Mathematical Journal, 56 (4) (2016),  540--551.


			\bibitem{ShishKlein} E. L. Shishkina; \emph{  Generalized Euler--Poisson--Darboux equation and			singular Klein--Gordon equation}. IOP Conf. Series: Journal of Physics: Conf. Series, 973 (2018), 1--21.
	
		\bibitem{ShSitPul}  E. L. Shishkina  and S. M. Sitnik; \emph{   General form of the Euler-Poisson-Darboux equation and application of the transmutation method}, Electronic Journal of Differential Equations, 2017 (177) (2017), 1--20.
		
		\bibitem{SitSh1}  E. L. Shishkina  and S. M. Sitnik; \emph{  On fractional powers of Bessel operators}, Journal of Inequalities and Special Functions,  Special issue To honor Prof. Ivan Dimovski's contributions, 8:1 (2017), 49--67.
		
\bibitem{SitSh2}  E. L. Shishkina  and S. M. Sitnik; \emph{On fractional powers of the Bessel operator on
semiaxis}, Siberian Electronic Mathematical Reports, 15 (2018), 1--10. 		
				
			\bibitem{Sit1} S. M. Sitnik; \emph{Factorization and norm estimation in weighted Lebesgue spaces of  Buschman--Erdelyi operators}, Doklades of the Soviet Academy of Sciences,  320  (6) (1991), 1326--1330.
			
				\bibitem{Sit2} S. M. Sitnik; \emph{Transmutations and applications: a survey}. 2010/12/16 	arXiv preprint arXiv:1012.3741, 141 P.

\bibitem{Sit3} S. M. Sitnik; \emph{A short survey of recent results on Buschman--Erdelyi transmutations},
Journal of Inequalities and Special Functions. (Special issue To honor Prof. Ivan Dimovski's contributions),
 8 (1) (2017),  140--157.

\bibitem{Sit4} S. M. Sitnik; \emph{Buschman--Erdelyi transmutations, classification and applications},
In the book: Analytic Methods Of Analysis   And Differential Equations: Amade 2012. ( Edited by M.V. Dubatovskaya,  S.V. Rogosin),
Cambridge Scientific Publishers, Cottenham, Cambridge, (2013), 171--201.
				
\bibitem{Sita1} S. M. Sitnik; \emph{Fractional integrodifferentiations for differential Bessel operator}, in:
					Proc. of the International Symposium ``The Equations of Mixed Type and Related Problems of the Analysis and Informatics",
					Nalchik (2004), 163--167.
					
\bibitem{Sita2} S. M. Sitnik; \emph{On explicit definitions of fractional powers of the Bessel differential operator and its applications to differential equations},
					Reports of the Adyghe (Circassian) International Academy of Sciences, 12 (2) (2010),  69--75.				
	
		\bibitem{Smirnov0}	M. M. Smirnov; \emph{ Problems in the equations of mathematical physics},  Moscow, Nauka, 1973.	
		
			\bibitem{Smirnov1}	M. M. Smirnov; \emph{Degenerate Hyperbolic Equations},  Minsk, 1977.
	
		\bibitem{Ida} I. G. Sprinkhuizen-Kuyper;  \emph{A fractional integral operator corresponding to negative powers of a certain second-order differential operator},  J. Math. Analysis and Applications, 72 (1979), 674--702.
	
		\bibitem{Sta}	 V. V. Stashevskaya; \emph{On the inverse problem of spectral analysis for a differential operator with a
		singularity at zero}, Zap. Mat. Otdel. Fiz.-Mat.Fak.KhGU i
		KhMO, 25 (4) (1957), 49--86.
		
			\bibitem{Tersenov} S. A. Tersenov; \emph{ Introduction in the theory of equations degenerating on a boundary}.  Novosibirsk state university, 1973.

		\bibitem{volk} V. Y. Volk; \emph{On inversion formulas for a
		differential equation with a singularity at $x=0$},
		Uspehi Matem. Nauk, 8 (4) (1953), 141--151.
		
			\bibitem{Watson} G. N. Watson; \emph{A Treatise on the Theory of Bessel Functions}, Cambridge University Press, 1966.
		
			\bibitem{Weinstein0} A. Weinstein; \emph{ On the wave equation and the equation of Euler--Poisson}, Proceedings of Symposia in Applied Mathematics,  V, Wave motion and vibration theory, McGraw-Hill Book Company, New York-Toronto-London, (1954), 137--147.
			
			\bibitem{Weinstein12} A. Weinstein; \emph{The generalized radiation problem and the Euler--Poisson--Darboux
			equation}, Summa Brasiliensis Math., 3 (7) (1955), 125--147.
			
			\bibitem{Weinstein13} A. Weinstein; \emph{On a Cauchy problem with subharmonic initial values}, Ann. Mat.
			Pura Appl. [IV] 43 (1957), 325--340.
	
	

	

	
	
\end{thebibliography}
\end{document}